\documentclass[reqno,oneside,12pt]{amsart}

\usepackage{fancyhdr}
\usepackage{xspace}
\usepackage{geometry}
\usepackage{tikz}
\usepackage{tikz-cd}
\usepackage{aligned-overset}
\tikzset{node distance=1.5cm, auto}
\usepackage{color}
\definecolor{darkgreen}{rgb}{0,0.45,0}
\usepackage[pagebackref,colorlinks,citecolor=darkgreen,linkcolor=darkgreen, urlcolor=darkgreen, hypertexnames=false]{hyperref}
\usepackage[capitalize]{cleveref}
\usepackage{stmaryrd}
\usepackage{enumerate}
\usepackage[shortlabels]{enumitem}
\usepackage{makecell}
\usepackage[justification = centering]{caption}
\usepackage{cancel}
\usepackage{comment}

\usepackage{graphicx}
\usepackage[all]{xy}
\usepackage{amsmath,amsfonts,amssymb}
\usepackage{mathtools,thmtools}

\usepackage{tikz}
\usepackage{tikz-cd}

\makeatletter
\newcommand{\tleft}{\mathrel\triangleleft}
\newcommand{\tright}{\mathrel\triangleright}
\DeclareRobustCommand{\btleft}{\mathrel{\mathpalette\btlr@\blacktriangleleft}}
\DeclareRobustCommand{\btright}{\mathrel{\mathpalette\btlr@\blacktriangleright}}

\newcommand{\btlr@}[2]{%
	\begingroup
	\sbox\z@{$\m@th#1\triangleright$}%
	\sbox\tw@{\resizebox{1.1\wd\z@}{1.1\ht\z@}{\raisebox{\depth}{$\m@th#1\mkern-1mu#2$}}}%
	\ht\tw@=\ht\z@ \dp\tw@=\dp\z@ \wd\tw@=\wd\z@
	\copy\tw@
	\endgroup
}

\usepackage{trimclip}
\newlength{\trianglewidth}
\newcommand{\Hom}{{\sf Hom}}

\newcommand{\End}{{\sf End}}

\newcommand{\Mod}{{\sf Mod}}

\newcommand{\Id}{{\sf Id}}

\newcommand{\PMod}{{\sf PMod}}
\newcommand{\HPMod}{{{_H}\sf PMod}}
\newcommand{\HMod}{{{_H}\sf Mod}}

\newcommand{\Vect}{{\sf Vect}}

\renewcommand{\Bbbk}{k}

\newcommand{\HYDH}{{_H} \mathcal{YD}^H}

\def\NN{{\mathbb N}}

\newcommand{\Bb}{\mathcal{B}}
\newcommand{\Cc}{\mathcal{C}}

\newcommand{\Mm}{\mathcal{M}}

\newcommand{\Pp}{\mathcal{P}}

\newcommand{\Tt}{\mathcal{T}}

\newtheorem{prop}{Proposition}[section]
\newtheorem{proposition}[prop]{Proposition}
\newtheorem{lemma}[prop]{Lemma} 
 
\newtheorem{corollary}[prop]{Corollary} 

\newtheorem{theorem}[prop]{Theorem}

\theoremstyle{definition}

\newtheorem{definition}[prop]{Definition}

\newtheorem{example}[prop]{Example}

\newtheorem{remark}[prop]{Remark}
\newtheorem{remarks}[prop]{Remarks}

\newcommand{\benu}{\begin{enumerate}}
	\newcommand{\enu}{\end{enumerate}}

\newcommand{\beqna}{\begin{eqnarray}}
	\newcommand{\eqna}{\end{eqnarray}}
\newcommand{\beqnast}{\begin{eqnarray*}}
	\newcommand{\eqnast}{\end{eqnarray*}}
\newcommand{\beqn}{\begin{equation}}
	\newcommand{\eqn}{\end{equation}}
\newcommand{\beqnst}{\begin{equation*}}
	\newcommand{\eqnst}{\end{equation*}}

\usepackage{latexsym}

\usepackage{xspace}
\usepackage{amscd}

\makeatother

\newcommand{\bema}{\left ( \begin{array}}
	\newcommand{\ema}{\end{array} \right )}

\newcommand{\HHom}{{_H}\mathsf{Hom}}
\newcommand{\HparHom}{{_{H_{par}}}\mathsf{Hom}}

\title{Globalization and the biactegory of partial modules}

\author[E. Batista]{Eliezer Batista}
\address{Departamento de Matem\'atica, Universidade Federal de Santa Catarina, Brazil}
\email{e.batista@ufsc.br}
\author[W. Hautekiet]{William Hautekiet}
\address{D\'epartement de Math\'ematique, Universit\'e Libre de Bruxelles, Belgium}
\email{william.hautekiet@ulb.be}
\author[J. Vercruysse]{Joost Vercruysse}
\address{D\'epartement de Math\'ematique, Universit\'e Libre de Bruxelles, Belgium}
\email{joost.vercruysse@ulb.be}


\begin{document}
	
	\begin{abstract}
		We show that the category of partial modules over a Hopf algebra \(H\) is a biactegory (a bimodule category) over the category of global \(H\)-modules. The     corresponding enrichment of partial modules over global modules is described, and the close relation between the dilation of partial modules and Hom-objects arising from this enrichment is investigated. In particular, for finite-dimensional pointed Hopf algebras, the standard dilation of a partial module \(M\) is isomorphic to the Hom-object from the monoidal unit to \(M\).
	\end{abstract}
	
	\maketitle
	\tableofcontents
	
	\section*{Introduction}
	
	Over the last two decades there has been in a highly increased interest in actegories (also called module categories) over monoidal categories (see e.g.\ \cite{EGNObook}). Such actegories can be viewed as a categorified version of actions or representations, similar to the way monoidal categories can be viewed as a categorification of groups and Hopf algebras. Just as the category of representations of a group provides valuable information on this group, understanding the actegories over a given monoidal category gives new insight in this monoidal category itself. Henceforth, it is important to see explicit constructions for such actegories, in particular for monoidal categories of representations of Hopf algebras. In this work we present a new source of examples of such actegories, arising from partial representation theory.
	
	Partial representations and partial actions of groups first appeared in the literature in the context of $C^\ast$-algebras \cite{ExelCircle}. The original problem was to characterize some $C^\ast$-algebras endowed with an action of the unit circle by $\ast$-automorphisms. These algebras are naturally graded by the additive group of integers, but fail to be isomorphic to a usual crossed product $A\rtimes \mathbb{Z}$. However, it was shown that these algebras are isomorphic to a {\em partial} crossed product, arising from a {\em partial} action of the group $\mathbb{Z}$. Given a partial action of a group $G$ on an algebra $A$, there is a canonical map from $G$ to the associated partial crossed product. In contrast to the classical case, this map is no longer a group homomorphism and the study of its behavior led to the notion of partial representation of a group on a Hilbert space in \cite{QuiggRaeburn}.
	Later, the notion has been refined and partial representations of groups came into the purely algebraic realm in \cite{DEP}, where it was shown that partial representations of a group $G$ coincide with usual representations of a particular algebra $\Bbbk_{par} G$, constructed from the group $G$. This algebra $\Bbbk_{par} G$ is isomorphic to a groupoid algebra, and has led to many subsequent developments (see \cite{Dok1,Dok2}). 
	
	One way to produce a partial action of a group $G$ is by restricting a global $G$-action on a set to an arbitrary subset. A question that arose very early in the theory was whether all partial actions could be obtained in this way. This is known as the globalization problem. A globalization for a partial action of a group on a set always exists \cite{KL}, and also for a partial action of a group on a $C^\ast$-algebra, the construction of a globalization is possible thanks to the existence of approximate units \cite{Abadieenveloping}. In the algebraic case, a partial action of a group $G$ on a unital algebra $A$ is globalizable if, and only if, every domain $A_g$ is an ideal of $A$ generated by a central idempotent in $A$ \cite{DE}. 
	The same question can be asked for partial representations of groups: when can a partial representation of a group be obtained by a restriction of a global representation of the same group? 
	This is called the problem of dilation. The specific motivating problem that led to dilations was the search for interaction groups that extended the actions of Ore semigroups \cite{Abadiedilations}.
	
	Partial actions of groups on algebras were extended to the Hopf algebraic framework in \cite{CJ}. Soon after, the globalization theorem for partial Hopf actions was proved in \cite{ABenveloping}. Partial representations of Hopf algebras were introduced in \cite{ABVparreps}, and a globalization result for these was obtained in \cite{ABVdilations}. As in the group case, there exists a universal algebra $H_{par}$, constructed out of the Hopf algebra $H$, which factorizes partial representations of $H$ by morphisms of algebras. Therefore, the theory of partial representations of a Hopf algebra $H$ is equivalent to the theory of  representations of the algebra $H_{par}$. This algebra has a rich structure, being a Hopf algebroid over a suitable base algebra $A_{par}\subseteq H_{par}$. One can show that $H$ acts partially on $A_{par}$ and that $H_{par}$ is isomorphic to the partial smash product $A_{par} \underline{\#} H$. In a certain sense, $A_{par}$ measures the ability of the Hopf algebra to produce partial representations. When the coradical of $H$ is trivial (e.\,g.\ $H$ is the universal enveloping algebra of a Lie algebra), then $A_{par}$ is trivial and $H$ admits no partial actions or representations other than the global ones \cite{AFHconnected}. On the other hand, in case of Sweedler's 4-dimensional Hopf algebra $H_4$, the algebra $A_{par}$ is infinite-dimensional.
	
	Since $H_{par}$ is a Hopf algebroid, its category of modules obtains a closed monoidal structure. This observation brought categorical techniques into the scope of the theory of partial representations. 
	The process of dilation (globalization) of partial modules establishes a categorical equivalence between the category of partial modules of $H$ and the category of $H$-modules equipped with a certain projection compatible with the action of $H$ .
	
	The main purpose of this present work is to deepen further our understanding of the category of partial modules of a Hopf algebra $H$. 
	In particular, we will show (see \cref{cor:biactegory}) that the category $\HPMod$ of partial $H$-modules, is a biactegory (or bimodule category) over the monoidal category $\HMod$ of (global) $H$-modules, by means of the $k$-linear tensor product. In other words, the $k$-linear tensor product of a partial and a global $H$-module is again a partial module over $H$. One should remark that this actegory structure differs strictly from the above mentioned monoidal structure on the category of partial $H$-modules, as the latter is built on the $A_{par}$-balanced tensor product.
	In Section 3, we show that the functors, given by tensoring global modules on the left or the right with a fixed partial module, allow right adjoints. This proves that the category $\HPMod$ can be enriched over the categories of $\HMod$ and $\HMod^{\mathrm{rev}}$ (see \cref{enrichment}). Section 4 is dedicated to the relationship between the Hom-objects arising from this enrichment and dilations of partial $H$-modules. For any partial $H$-module $M$, there in a natural injective morphism of $H$-modules between its so-called standard dilation $\overline M$ and the Hom-object $\{ A_{par} ,M \}$ (\cref{cor:Xi}). We show that for a pointed Hopf algebra with a finite group of grouplike elements, this injective map is in fact an isomorphism (\cref{th:Xi_iso_pointed}). This reveals a deep relation between the globalization question for partial representations and the actegory structure that we introduce here.
	Finally, for a finite-dimensional Hopf algebra $H$, we define a new Hopf algebroid $H_{glob}= \{ A_{par} ,A_{par} \} \# H$ in Section 5. We prove in \cref{H_par_Morita_H_glob} that for a finite dimensional pointed Hopf algebra $H$,  $H_{glob}$ is Morita equivalent to $H_{par}\cong A_{par}\underline{\#} H$, using a result about Morita equivalence between partial smash products and global smash products involving the globalization from \cite{ABenveloping}. The Hopf algebroid \(H_{glob}\) generalizes the groupoid algebra $\Bbbk_{glob}G$ proposed in \cite{Velascothesis} obtained from the globalization of partial actions of a finite group $G$: there is a sequence of algebra isomorphisms
	\[
	\Bbbk_{glob} G\cong \overline{A_{par}(G)}\# \Bbbk G \cong \{ A_{par} (G) , A_{par} (G) \} \# \Bbbk G \cong (\Bbbk G)_{glob}.
	\]
	
	\section{Preliminaries and first results}
	
	Throughout this article, $H$ is a Hopf algebra over a field $k$ with invertible antipode \(S\). We will adopt the Sweedler notation for its comultiplication, writing \(\Delta(h) = h_{(1)} \otimes h_{(2)}\) for any \(h \in H\). 
	The action of a Hopf algebra \(H\) on a (global) left module \(X\) will be generally denoted by \(\btright \ : H \otimes X \to X\). If \(X\) is right \(H\)-module and \(V\) is a vector space, then \(\Hom_k(X, V)\) has a natural \(H\)-module structure that we denote by \(\rightharpoonup\). It is defined by \((h \rightharpoonup f)(x) = f(x \btleft h)\) for \(h \in H, f \in \Hom_k(X, V)\) and \(x \in X\).
	
	\subsection{Partial modules and partial representations}\label{se:partialprelim}
	
	\begin{definition}[\cite{ABVparreps}]
		A partial representation of $H$ is a pair \((B, \pi)\), where \(B\) is a $k$-algebra and \(\pi\) is a $k$-linear map $H \to B$ satisfying the following properties:
		\begin{enumerate}[(PR1), leftmargin = 2 cm]
			\item $\pi(1_H) = 1_B;$ \label{PR1}
			\item \(\pi(h)\pi(k_{(1)}) \pi(S(k_{(2)})) = \pi(hk_{(1)}) \pi(S(k_{(2)})),\) for all \(h, k \in H\); \label{PR2}
			\item \(\pi(h_{(1)}) \pi(S(h_{(2)})) \pi(k) = \pi(h_{(1)}) \pi(S(h_{(2)})k),\) for all \(h, k \in H\). \label{PR3}
		\end{enumerate}
	\end{definition}
	
	By Saracco's Lemma (see \cite[Lemma 2.11]{ABCQVparcorep}), axioms \ref{PR2} and \ref{PR3} can be replaced by
	\begin{enumerate}[(PR1), leftmargin = 2 cm]
		\setcounter{enumi}{3}
		\item \(\pi(h) \pi(S(k_{(1)})) \pi(k_{(2)}) = \pi(hS(k_{(1)})) \pi(k_{(2)})\), for all \(h,k\in H\); \label{PR4}
		\item \(\pi(S(h_{(1)})) \pi(h_{(2)}) \pi(k) = \pi(S(h_{(1)})) \pi(h_{(2)}k)\), for all $h,k\in H$. \label{PR5}
	\end{enumerate}
	If \(B = \End_k(M)\), then we say that \(M\) is a partial \(H\)-module and we write
	\[h \bullet m = \pi(h)(m)\]
	for \(h \in H, m \in M\).
	Throughout the paper we will denote partial action by the bullet \(\bullet\) in order to distinguish them from global actions, which are denoted by \(\btright\) or \(\rightharpoonup\), as mentioned above.
	
	A morphism between partial $H$-modules $M$ and $N$ is a $k$-linear map $f:M\rightarrow N$ such that
	\[
	f(h\bullet m)=h\bullet f(m), \qquad \forall h\in H, \; m\in M
	\]
	The category of partial $H$-modules will be denoted by ${}_H\PMod$.
	
	By quotienting the tensor algebra over the vector space $H$ by the relations defining a partial representation, one obtains a new algebra $H_{par}$ along with a partial representation $[-]:H\to H_{par}$ with the universal property that
	any partial representation $H\to B$ factors through an algebra map $H_{par}\to B$, in other words,
	partial representations of $H$ coincide with (global) representations of $H_{par}$ (\cite[Theorem 4.2]{ABVparreps}). 
	Hence we have the following identities, which in fact define $H_{par}$:
	$$[1_H]=1_{H_{par}}, \quad
	[h][k_{(1)}][S(k_{(2)})]=[hk_{(1)}][S(k_{(2)})], \quad
	[h_{(1)}][S(h_{(2)})][k]=[h_{(1)}][S(h_{(2)})k].$$
	Although $H_{par}$ is often called the ``partial Hopf algebra'' of $H$, it is in fact not a Hopf algebra, but rather a Hopf algebroid over the base algebra $A_{par}$, which is the subalgebra of $H_{par}$ generated by elements of the form
	\[
	\varepsilon_h =[h_{(1)}][S(h_{(2)})] \in H_{par}.
	\]
	The Hopf algebroid structure of $H_{par}$ turns the category of partial $H$-modules into a monoidal category, whose monoidal product is $\otimes_{A_{par}}$ and whose monoidal unit is $A_{par}$. The \(A_{par}\)-bimodule structure on a partial \(H\)-module \(M\) is given by
	\begin{align}
		\varepsilon_h \cdot m &= h_{(1)} \bullet (S(h_{(2)}) \bullet m), \label{eq:leftAaction} \\
		m \cdot \varepsilon_h &= h_{(2)} \bullet (S^{-1}(h_{(1)}) \bullet m). \label{eq:rightAaction}
	\end{align}
	
	\begin{example}[\cite{DEP}] 
		\label{ex:group}
		Partial representations of a group algebra \(kG\) are exactly the so-called partial representations of the group $G$. The algebra $(kG)_{par}$ is usually denoted by $k_{par} G$, generated by symbols $[g]$, for elements $g\in G$ and satisfying 
		\[
		[e]=1_{k_{par}G}, \qquad [g][h][h^{-1}]=[gh][h^{-1}], \qquad [g^{-1}][g][h]=[g^{-1}][gh].
		\]
		In this case, the elements $\varepsilon_g =[g][g^{-1}]$ are idempotents and commute among each other.
	\end{example}
	
	Since the category of partial $H$-modules is monoidal, one can consider the algebra objects inside that category. These are called the partial actions of $H$. Historically, the notion of a partial $H$-action appeared before that of a partial $H$-module \cite{ABenveloping, CJ}. 
	
	\begin{definition}
		A symmetric partial action of a Hopf algebra $H$ on a unital algebra $A$ is a $k$-linear map $\bullet :H\otimes A\rightarrow A$ such that
		\begin{enumerate}
			\item[(PA1)] $1_H \bullet a=a$, for all $a\in A$.
			\item[(PA2)] $h\bullet (ab)=(h_{(1)} \bullet a)(h_{(2)} \bullet b)$, for all $h\in H$, $a,b\in A$.
			\item[(PA3)] $h\bullet (k \bullet a)=(h_{(1)} \bullet 1_A )(h_{(2)}k \bullet a)=(h_{(1)}k \bullet a)(h_{(2)} \bullet 1_A )$, for all $h,k\in H$, $a\in A$. 
		\end{enumerate}
		The algebra $A$ is also called a partial $H$-module algebra.
	\end{definition}
	
	Given a partial $H$-module algebra $A$, one can define a multiplication on $A\otimes H$ given by
	\begin{equation}
		(a\otimes h)(b\otimes k)=a(h_{(1)}\bullet b)\otimes h_{(2)}b. \label{eq:multiplication_smash}
	\end{equation}
	This multiplication is associative, but $A\otimes H$ need not contain a unit element. Consider the subspace
	\[
	A\underline{\#} H=(A\otimes H)(1_A \otimes 1_H),
	\]
	which is generated by elements of the type
	\[
	a\# h=a(h_{(1)}\bullet 1_A)\otimes h_{(2)} .
	\]
	Now this defines a unital algebra with the multiplication \eqref{eq:multiplication_smash}, called the \textit{partial smash product} $A\underline{\#} H$ in \cite{ABenveloping}. For every $a\in A$ and $h\in H$, we have the identity 
	\begin{equation}
		\label{eq:smash_rewriting}
		a(h_{(1)} \bullet 1_{A_{par}}) \# h_{(2)} = a(h_{(1)} \bullet 1_{A_{par}})(h_{(2)} \bullet 1_{A_{par}}) \otimes h_{(3)} = a(h_{(1)} \bullet 1_{A_{par}}) \otimes h_{(2)} =  a \# h.
	\end{equation}
	The algebra \(A_{par}\) is in fact a partial module algebra; the partial action of \(H\) on \(A_{par}\) is given by
	\begin{equation}
		\label{eq:partial_module_Apar}
		h\bullet \varepsilon_{k^1} \cdots \varepsilon_{k^n} =[h_{(1)}]\varepsilon_{k^1} \cdots \varepsilon_{k^n} [S(h_{(2)})]=\varepsilon_{h_{(1)}k^1} \cdots \varepsilon_{h_{(n)}k^n} \varepsilon_{h_{(n+1)}}.
	\end{equation}
	The partial Hopf algebra $H_{par}$ is then isomorphic to the partial smash product $A_{par} \underline{\#} H$ by means of the map
	\begin{equation}	
		\label{eq:sigma}
		\sigma : A_{par}\underline{\#} H \to H_{par}, \quad a \# h \mapsto a[h].
	\end{equation}

	For more information and proofs of the results mentioned here, we refer to \cite{ABVparreps}.
	
	\subsection{Dilations of partial modules and globalizations of partial actions}
	\label{se:dilation}
	
	It is possible to construct partial $H$-modules out of global ones by using projections \cite{ABVdilations}. More precisely, given an $H$-module $N$, we say that a $k$-linear projection $T:N\rightarrow N$ satisfies the so called \emph{c-condition} if for every element $h\in H$, we have
	\begin{equation}
		T\circ T_h =T_h \circ T, \label{eq:ccondition}
	\end{equation}
	where
	\(
	T_h (n) =h_{(1)} \btright T(S(h_{(2)})\btright n) .
	\)
	In this case, the image $T(N)$ can be made into a partial $H$-module by
	\begin{equation}
		\label{eq:TM}
		h\bullet T(n)=T(h\btright T(n)).
	\end{equation}
	
	\begin{definition}[{\cite{ABVdilations}}]\label{defofdilation}
		A \emph{dilation} of a partial $H$-module $M$ is a triple $(N,T,\theta )$ in which
		\begin{enumerate}[(D{i}1), leftmargin = 2 cm]
			\item $N$ is an $H$-module;
			\item $T=T^2:N\rightarrow N$ is a projection satisfying the c-condition \eqref{eq:ccondition};
			\item $\theta :M\rightarrow T(N)$ is an isomorphism of partial $H$-modules. 
		\end{enumerate}
		We say that a dilation $(N,T, \theta)$ is proper if $N$ is generated by $T(N)=\theta (M)$ as an $H$-module and that this dilation is minimal if $N$ doesn't contain any nontrivial submodule that is annihilated by $T$. 
		
		A morphism between two dilations $(N_1 ,T_1 , \theta_1)$ and $(N_2 ,T_2 ,\theta_2 )$ of a partial $H$-module $M$ is a morphism of $H$-modules $f: N_1 \rightarrow N_2$ such that \(T_2 \circ f=f\circ T_1\)
		and	\(f\circ \theta_1 =\theta_2\).
	\end{definition}
	
	\begin{theorem}[{\cite[Theorem 4.3]{ABVdilations}}]
		\label{th:standard_dilation}
		Every partial $H$-module $M$ admits a proper and minimal dilation $(\overline{M} , \overline{T} , \varphi )$, called the standard dilation.
	\end{theorem}
	
	For sake of completeness, let us recall the construction of the standard dilation of a partial $H$-module $M$.
	\begin{enumerate}
		\item Consider the natural $H$-module structure on $\Hom_k (H,M)$ given by $(h \rightharpoonup f)(k)=f(kh)$.
		\item Define the $k$-linear map $\varphi :M\rightarrow \Hom_k (H,M)$ by $\varphi (m)(h)=h\bullet m$. It is easy to see that $\varphi$ is injective, so we have an isomorphic copy of $M$ as a subspace of $\Hom_k (H,M)$.
		\item Define $\overline{M} =H\rightharpoonup \varphi (M)$, the $H$-submodule generated by $\varphi (M)$. A typical element of $\overline{M}$ is of the form
		\[
		f=\sum_i h_i \rightharpoonup \varphi (m_i ),
		\]
		which, as a map from $H$ to $M$, is explicitly written as
		\[
		f(k) =\sum_i kh_i \bullet m_i.
		\]
		\item The projection $\overline{T}:\overline{M} \rightarrow \overline{M}$ is given by
		\[
		\overline{T} (f) =\varphi (f(1_H)).
		\]
	\end{enumerate}
	
	The standard dilation defines an equivalence \(\overline{D}\) between $\HPMod$ and the category $\mathbb{T}(\HMod )$, whose objects are pairs $(N,T)$ with $N$ a global $H$-module and $T:N\rightarrow N$ a linear projection satisfying the c-condition, and whose morphisms are $H$-module maps intertwining their respective projections. The inverse equivalence is given by the restriction functor, sending a pair \((N, T)\) in $\mathbb{T}(\HMod )$ to the partial module \(T(N)\). The composition of \(\overline{D}\) with the forgetful functor \(\mathbb{T}(\HMod ) \to \HMod\) is called the \textit{dilation functor} and is denoted by \(D\).
	
	One partial $H$-module can admit several distinct dilations whose underlying global $H$-modules are not isomorphic, but the standard dilation is (up to isomorphism) the unique minimal and proper dilation associated to each partial $H$-module.
	This is because the standard dilation satisfies two universal properties. The first is \cite[Proposition 4.7]{ABVdilations} and says that for any proper dilation \((N, T, \theta)\) of a partial module \(M,\) there is a unique surjective \(H\)-linear map \(\Phi : N \to \overline{M}\) such \(\overline{T} \circ \Phi = \Phi \circ T\) and \(\Phi \circ \theta = \varphi\). The second universal property is with respect to minimal dilations and is formulated in the next proposition. It did not appear in the literature before.
	
	\begin{proposition}
		\label{prop:universal_minimal}
		Let \(M\) be a partial \(H\)-module and \((\overline{M}, \overline{T}, \varphi)\) its standard dilation. For any minimal dilation \((N, T, \theta)\) of \(M,\) there exists a unique injective \(H\)-linear map \(\Lambda : \overline{M} \to N\) such that \(T \circ \Lambda = \Lambda \circ \overline{T}\) and \(\Lambda \circ \varphi = \theta\), which is bijective if and only if \((N, T, \theta)\) is proper.
	\end{proposition}
	\begin{proof}
		For any \(x = \sum_i h_i \rightharpoonup \varphi(m_i) \in \overline{M},\) define
		\[\Lambda(x) = \sum_i h_i \btright \theta(m_i) \in N.\]
		Let us show that the map \(\Lambda : \overline{M} \to N\) is well-defined. If \(h'_j \in H\) and \(m_j' \in M\) are such that also \(x = \sum_j h'_j \rightharpoonup \varphi(m_j')\), then for any \(l \in H\)
		\[\sum_i lh_i \bullet m_i = \sum_j lh_j' \bullet m_j',\]
		hence
		\[T\left(\sum_i lh_i \btright \theta(m_i)\right) = \theta\left(\sum_i lh_i \bullet m_i\right) = \theta\left(\sum_j lh'_j \bullet m'_j\right) = T\left(\sum_j lh_j' \btright \theta(m_j') \right)\]
		because \(\theta\) is a morphism of partial modules, and the partial module structure on \(T(N)\) is given by \eqref{eq:TM}. This shows that the \(H\)-submodule of \(N\) generated by
		\[\sum_i h_i \btright \theta(m_i) - \sum_k h'_j \btright \theta(m'_j)\]
		is annihilated by \(T,\) so it is zero since \((N, T, \theta)\) is a minimal dilation. 
		
		It is clear that \(\Lambda\) is \(H\)-linear and that \(\Lambda \circ \varphi = \theta\). Moreover, \(\Lambda\) is the unique map \(\overline{M} \to N\) with these properties. Let us show that \(\Lambda\) is injective. For any \(x = \sum_i h_i \rightharpoonup \varphi(m_i) \in \overline{M},\)
		\begin{equation*}
			T(\Lambda(x)) = T\left(\sum_i h_i \btright \theta(m_i)\right) = \theta\left(\sum_i h_i \bullet m_i\right)
			= (\Lambda\circ \varphi)\left(\sum_i h_i \bullet m_i\right) = \Lambda(\overline{T}(x)),
		\end{equation*}
		and if \(\Lambda(x) = 0\), then for any \(l \in H,\)
		\[\theta\left(\sum_i lh_i \bullet m_i\right) = T\left( \sum_i lh_i \btright \theta(m_i)\right) = 0.\]
		But \(\theta\) is an isomorphism, so this implies that
		\[\left(\sum_i h_i \rightharpoonup \varphi(m_i)\right)(l) = \sum_i lh_i \bullet m_i = 0.\]
		
		If \((N, T, \theta)\) is moreover proper, then any element of \(N\) is of the form \(\sum_i h_i \btright \theta(m_i)\), so then \(\Lambda\) is clearly surjective too.
	\end{proof}
	
	If \(M = A\) is a partial \(H\)-module algebra, we speak about \textit{globalizations} (or enveloping actions) rather than dilations. Globalizations of partial actions appeared in the literature before dilations of partial modules (see \cite{ABenveloping}). We reformulate \cite[Definition 3]{ABenveloping} below.
	
	\begin{definition}
		\label{def:globalization}
		Let \(A\) be a partial \(H\)-module algebra. A \textit{globalization} of \(A\) is a pair \((B, \theta)\), where \(B\) is a (not necessarily unital) \(H\)-module algebra and \(\theta : A \to B\) is an injective and multiplicative map, such that
		\begin{enumerate}[(i)]
			\item \(\theta(A)\) is a right ideal in \(B\),
			\item the triple \((B, T_\theta, \theta)\) is a dilation of the partial module \(A\), where \(T_\theta\) is the projection
			\[T_\theta : B \to B, \quad b \mapsto \theta(1_A) b.\]
		\end{enumerate}
		As for dilations, the globalization is called \textit{proper} if \(B\) is generated by \(\theta(A)\) as an \(H\)-module.
	\end{definition}
	The construction given after \cref{th:standard_dilation} applied to a partial module algebra \(A\) provides a globalization of \(A,\) which is called the standard globalization. Indeed, since 
	\begin{equation}
		\label{eq:multiplication_globalization}
		(h\rightharpoonup \varphi (a)) * (k\rightharpoonup \varphi (b))=h_{(1)} \rightharpoonup \varphi (a(S(h_{(2)})k \bullet b))
	\end{equation}
	(see \cite[Lemma 2]{ABenveloping}), the convolution product on $\Hom_k (H,A)$ restricts to a well-defined composition on the submodule $H\rightharpoonup \varphi (A)$, turning it into an $H$-module algebra. Moreover, it is clear from \eqref{eq:multiplication_globalization} that \(\varphi : A \to \Hom_k (H,A)\) is multiplicative and that its image is a right ideal.

	\subsection{Actegories and biactegories} 
	
	Actegories are the categorification of modules: a monoidal category acts on a category by means of an action functor. The following definitions can be found in \cite{CGamthematician} and \cite{JKactegories}.
	
	\begin{definition}
		\label{def:actegory}
		Let \((\mathcal{C}, \otimes, I, a, \ell, r)\) be a monoidal category. A left action of \(\mathcal{C}\) on a category \(\mathcal{M}\) consists of 
		\begin{enumerate}[(i)]
			\item a functor \(\tright\, : \mathcal{C} \times \mathcal{M} \to \mathcal{M}\);
			\item a natural isomorphism \(\lambda_M : I \tright M \to M\) called \textit{unitor};
			\item a natural isomorphism \(\mu_{X, Y, M} : (X \otimes Y) \tright M \to X \tright (Y \tright M)\) called \textit{multiplicator},
		\end{enumerate}
		such that 
		\begin{align*}
			\lambda_{X \tright M} \circ \mu_{I, X, M} &= \ell_X \tright M, \\
			(X \tright \lambda_M) \circ \mu_{X, I, M} &= r_X \tright M, \\
			(X \tright \mu_{Y, Z, M}) \circ \mu_{X, Y \otimes Z, M} &= \mu_{X, Y, Z \tright M} \circ \mu_{X \otimes Y, Z, M} \circ (a_{X, Y, Z} \tright M).
		\end{align*}
		
		This is in fact equivalent to having a strong monoidal functor \(\mathcal{C} \to [\mathcal{M}, \mathcal{M}]\) (it is given by \(X \mapsto X \tright -\)). We say that \(\mathcal{M}\) is a left \(\mathcal{C}\)-\textit{actegory}\footnote{In the literature, also the term {\em $\Cc$-module category} is used. However, we prefer to avoid this terminology in order not to confuse `module category' with `category of modules'.}. Right \(\mathcal{C}\)-actegories are defined using right actions of \(\mathcal{C}\) on \(\mathcal{M}\) and correspond to strong monoidal functors \(\mathcal{C} \to [\mathcal{M}, \mathcal{M}]^{\mathrm{rev}}\), where \([\mathcal{M}, \mathcal{M}]^{\mathrm{rev}}\) is the same category as \([\mathcal{M}, \mathcal{M}]\), but with the reversed monoidal product.
		
		A \(\Cc\)-biactegory is a category \(\Mm\) equipped with a left action \(\tright\) and a right action \(\tleft\) of \(\Cc\) together with a natural isomorphism, called \textit{bimodulator},
		\[\zeta_{X, M, Y} : X \tright (M \tleft Y) \to (X \tright M) \tleft Y,\]
		satisfying suitable coherence axioms.
	\end{definition}
	
	\begin{definition}
		\begin{enumerate}[(i)]
			\item A (strong) \(\mathcal{C}\)-linear functor between left \(\mathcal{C}\)-actegories \((\mathcal{M}, \tright)\) and \((\mathcal{N}, \diamond)\) is a functor \(F : \mathcal{M} \to \mathcal{N}\) equipped with a natural isomorphism (called  \textit{lineator})
			\[\ell_{X, M} : X \diamond F(M) \to F(X \tright M)\]
			compatible with the unitors and multiplicators. If all \(\ell_{X, M}\) are identity morphisms, \(F\) is called strict.
			\item A \(\mathcal{C}\)-linear natural transformation between \(\mathcal{C}\)-linear functors \((F, \ell)\) and \((G, \ell')\) is a natural transformation \(\alpha : F \Rightarrow G\) such that
			\[\begin{tikzcd}
				X \diamond F(M) \arrow{r}{X \diamond \alpha_M} \arrow{d}{\ell_{X, M}} & X \diamond G(M) \arrow{d}{\ell'_{X, M}} \\
				F(X \tright M) \arrow{r}{\alpha_{X \tright M}} & G(X \tright M).
			\end{tikzcd}\]
		\end{enumerate}
	\end{definition}
	
	For a \(k\)-bialgebra \(B\), the category \({_B}\Mod\) acts on \(\Vect_k\); this is the restriction of the regular action of \(\Vect_k\) on itself to \(\otimes : {_B}\Mod \times \Vect_k \to \Vect_k\). The following reconstruction result is well-known, but we provide a proof for sake of completeness.
	
	\begin{proposition}
		\label{prop:reconstruction}
		Let \(B\) be a \(k\)-bialgebra and \(A\) be a \(k\)-algebra. Then there is a one-to-one correspondence between
		\begin{enumerate}[(i)]
			\item left \(B\)-comodule algebra structures on \(A\);
			\item strong monoidal functors of the form
			\begin{equation}
				\label{eq:BModaction}
				{_B} \Mod \to [{_A}\Mod, {_A}\Mod],\quad X \mapsto (X \otimes -),
			\end{equation}
			\item \({_B}\Mod\)-actegory structures on \({_A}\Mod\) such that the forgetful functor \({_A}\Mod \to \Vect_k\) is strict \({_B}\Mod\)-linear. 
		\end{enumerate}
	\end{proposition}
	\begin{proof}
		\par{(i) \(\Rightarrow\) (ii).} If \(A\) is a left \(B\)-comodule algebra with coaction \[\rho : A \to B \otimes A,\ \rho(a) = a^{(-1)} \otimes a^{(0)},\] then for every \(B\)-module \(X\) and \(A\)-module \(M\), \(X \otimes M\) is again an \(A\)-module with action \\ \(a \cdot (x \otimes m) = a^{(-1)} \cdot x \otimes a^{(0)} \cdot m\) (the module axioms are satisfied because \(\rho(1_A) = 1_B \otimes 1_A\) and \(\rho(aa') = \rho(a)\rho(a')\) for all \(a, a' \in A\)). Hence the functor \({_B} \Mod \to [{_A}\Mod, {_A}\Mod],\  X \mapsto (X \otimes -)\) is well-defined and it is strong monoidal thanks to the comodule axioms satisfied by \(\rho\).
		
		\par{(ii) \(\Rightarrow\) (iii).} The functor \eqref{eq:BModaction} makes \({_A}\Mod\) a \({_B}\Mod\)-actegory, with the unitor and multiplicator coming from the strong monoidal structure. It is clear that the forgetful functor to \(\Vect_k\) is strict \({_B}\Mod\)-linear. 
		
		\par{(iii) \(\Rightarrow\) (i).} If \({_A}\Mod\) is a \({_B}\Mod\)-actegory such that the forgetful functor \({_A}\Mod \to \Vect_k\) is strict \({_B}\Mod\)-linear, then the associated strong monoidal functor \({_B} \Mod \to [{_A}\Mod, {_A}\Mod]\) is of the form \eqref{eq:BModaction}. In particular \(B \otimes A\) is an \(A\)-module, and \(\rho(a) \coloneqq a \cdot (1_B \otimes 1_A)\) defines a \(B\)-comodule algebra structure on \(A\).
	\end{proof}
	
	Of course there are similar statements for right actegories and biactegories, corresponding to right comodule algebras and bicomodule algebras, respectively. Similarly, a category of comodules over a coalgebra \(C\) is an actegory (for \(\otimes\)) over the category of comodules over a bialgebra \(B\) precisely if \(C\) is a \(B\)-module coalgebra.

	\section{The biactegory of partial modules}
	\label{se:biact}
	
	In \cite[Lemma 5.5]{ABVdilations}, it was shown that the tensor product \(M \otimes X\) of a partial module \(M\) and a global module \(X\) is again a partial module if endowed with the diagonal action
	\begin{equation}\label{actegory1}
		h \bullet (m \otimes x) = h_{(1)} \bullet m \otimes h_{(2)} \btright x.
	\end{equation}
	It turns out that also \(X \otimes M\) is a partial module for the diagonal action
	\begin{equation}\label{actegory2}
		h \bullet (x \otimes m) = h_{(1)} \btright x \otimes h_{(2)} \bullet m
	\end{equation}
	and that in this way \(\HPMod\) is a biactegory over \(\HMod\). We will prove this in detail by taking advantage of \cref{prop:reconstruction}.

	\begin{lemma}
		\label{le:bicomodule}
		Let \(H\) be a Hopf algebra. The partial Hopf algebra \(H_{par}\) is an \(H\)-bicomodule algebra, by means of 
		\begin{align}
			H_{par} \to H \otimes H_{par},\quad [h] \mapsto h_{(1)} \otimes [h_{(2)}], \label{eq:leftHcomodule}  \\
			H_{par} \to H_{par} \otimes H,\quad [h] \mapsto [h_{(1)}] \otimes h_{(2)}, \label{eq:rightHcomodule}
		\end{align}
		extended multiplicatively.
	\end{lemma}
	\begin{proof}	
		We start by showing that \eqref{eq:rightHcomodule} is well-defined. Consider the algebra \(H_{par} \otimes H\) and the map
		\[\pi : H \to H_{par} \otimes H, \quad h \mapsto [h_{(1)}] \otimes h_{(2)}.\]
		Then first of all \(\pi(1_H) = [1_H] \otimes 1_H = 1_{H_{par}} \otimes 1_H\). Therefore, the axiom \ref{PR1} for partial representations is verified. Moreover, for any \(h, k \in H,\)
		\begin{align*}
			\pi(h) \pi(k_{(1)}) \pi(S(k_{(2)})) &= [h_{(1)}][k_{(1)}][S(k_{(4)})] \otimes h_{(2)} k_{(2)} S(k_{(3)})
			= [h_{(1)}][k_{(1)}][S(k_{(2)})] \otimes h_{(2)} \\
			&= [h_{(1)}k_{(1)}] [S(k_{(2)})] \otimes h_{(2)}
			= [h_{(1)}k_{(1)}] [S(k_{(4)})] \otimes h_{(2)} k_{(2)} S(k_{(3)}) \\
			&= \pi(hk_{(1)}) \pi(S(k_{(2)})) .
		\end{align*}
		Hence, the axiom \ref{PR2} for partial representations is verified. Similarly, one can check the axiom \ref{PR3}
		\[\pi(h_{(1)}) \pi(S(h_{(2)})) \pi(k) = \pi(h_{(1)}) \pi(S(h_{(2)})k).\]
		By the universal property of the partial Hopf algebra $H_{par}$ (see Section \ref{se:partialprelim}), we obtain an algebra map \(H_{par} \to H_{par} \otimes H\) which is given by 
		\begin{equation}
			\label{eq:Hcoaction}
			[h] \mapsto [h_{(1)}] \otimes h_{(2)}.
		\end{equation}
		It is easy to check the counitality and coassociativity, so \(H_{par}\) is indeed a right \(H\)-comodule algebra. To show that \eqref{eq:leftHcomodule} is well-defined and endows \(H_{par}\) with a left comodule algebra structure, one needs to use the alternative axioms \ref{PR4} and \ref{PR5} for partial representations. The compatibility between the left and right comodule structure is clear, so \(H_{par}\) is an \(H\)-bicomodule algebra. 
	\end{proof}

	Combining the biactegory counterpart of \cref{prop:reconstruction} with \cref{le:bicomodule}, we obtain that \(\HPMod\) is an \(\HMod\)-biactegory. 
	
	\begin{theorem} 
		\label{cor:biactegory}
		Let \(H\) be a Hopf algebra. The monoidal structure on $\Vect_k$ lifts strictly to actions
		\begin{align}
			\otimes : {_H}\mathsf{Mod} \times {_H}\mathsf{PMod} \to {_H}\mathsf{PMod}, \label{eq:leftactegory} \\
			\otimes : {_H}\mathsf{PMod} \times {_H}\mathsf{Mod} \to {_H}\mathsf{PMod}, \label{eq:rightactegory} 			 
		\end{align}
		in the sense of \cref{def:actegory}, which make \(\HPMod\) a biactegory over \(\HMod\).
	\end{theorem}

	For an \(H\)-module \(X\) and a partial \(H\)-module \(M\), the partial \(H\)-module structure on \(M \otimes X\) is given by \eqref{actegory1} and on \(X\otimes M\) by \eqref{actegory2}.

	\begin{remarks}  
		\begin{enumerate}[(i)]
			\item
			\cref{cor:biactegory} can be adapted to partial comodules (see \cite{ABCQVparcorep} for the relevant definitions). The monoidal product on $\Vect_k$ lifts to left and right actions of the monoidal category of right $H$-comodules $\mathsf{Mod}^H$ on the category of right partial $H$-comodules $\mathsf{PMod}^H$:
			\begin{align*}
				\otimes : \mathsf{Mod}^H \times \mathsf{PMod}^H \to \mathsf{PMod}^H, \\
				\otimes : \mathsf{PMod}^H \times \mathsf{Mod}^H \to \mathsf{PMod}^H.
			\end{align*}
			\item
			We know that \({_H}\mathsf{PMod}\) is monoidal for the monoidal product \(\otimes_{A_{par}}\) and that \({_H}\mathsf{Mod}\) is a full subcategory of \({_H}\mathsf{PMod}\). Moreover, the monoidal product \(\otimes_{A_{par}}\) restricted to \({_H}\mathsf{Mod}\) is exactly \(\otimes\). Indeed, if \(X\) and \(Y\) are global \(H\)-modules, then \(A_{par}\) acts trivially on them, i.\,e.\ 
			\[\varepsilon_h \cdot x \cdot \varepsilon_k = \epsilon(h)\epsilon(k) x\]
			for all \(h, k \in H, x \in X\). It follows that
			\[x \cdot \varepsilon_h \otimes y = x \epsilon(h) \otimes y = x \otimes \epsilon(h)y = x \otimes \varepsilon_h \cdot y\]
			and hence \(X \otimes_{A_{par}} Y \cong X \otimes Y\).
			
			However, the monoidal action \eqref{eq:rightactegory}
			does not coincide with the restriction of the monoidal product
			\[\otimes_{A_{par}} : {_H}\mathsf{PMod} \times {_H}\mathsf{PMod} \to {_H}\mathsf{PMod}.\]
			Indeed, recall that \(A_{par}\) is itself a left partial \(H\)-module with \(h \bullet a = [h_{(1)}] a [S(h_{(2)})]\) (see Section~\ref{se:partialprelim}) and that by applying the forgetful functor \({_{H_{par}}} \Mod \to {_{A_{par}}}\Mod_{A_{par}}\) (\ref{eq:leftAaction}, \ref{eq:rightAaction}), the regular \(A_{par}\)-bimodule is obtained. If we consider the trivial \(H\)-module \(\Bbbk\) for which \(h \btright 1 = \epsilon(h)\), then 
			\begin{align*}
				A_{par} \otimes_{A_{par}} \Bbbk \cong \Bbbk, \\
				A_{par} \otimes \Bbbk \cong A_{par}.
			\end{align*}
		\end{enumerate}
	\end{remarks}

	\section{Partial modules are enriched over global modules}
	
	In the previous section we showed that \(\HPMod\) is a biactegory over \(\HMod\). In particular, for any global \(H\)-module \(X\) there are functors
	\begin{align}
		X \otimes - &: \HPMod \to \HPMod, \label{eq:X-} \\
		- \otimes X &: \HPMod \to \HPMod, \label{eq:-X}
	\end{align}
	and for any partial \(H\)-module \(M\) there are functors
	\begin{align}
		- \otimes M &: \HMod \to \HPMod, \label{eq:-M}\\
		M \otimes - &: \HMod \to \HPMod. \label{eq:M-}
	\end{align}
	We will show that each of these has a right adjoint. For \eqref{eq:X-} and \eqref{eq:-X}, this is just the closed structure of \(\HMod\) lifted to an adjunction on \(\HPMod\). The adjoints of \eqref{eq:-M} and \eqref{eq:M-} provide enrichments of \(\HPMod\) over \(\HMod,\) as developed in the setting of tensor categories in \cite[\S 7.9]{EGNObook}.
	
	\begin{lemma}
		Let \(X\) be a global \(H\)-module and \(M\) a partial \(H\)-module. 
		\begin{enumerate}[(i)]
			\item Let \(\Hom_k^r(X, M) = \Hom_k(X, M)\) as vector spaces, and put 
			\begin{equation}
				(h \cdot f)(x) = h_{(1)} \bullet f(S(h_{(2)}) \btright x)
			\end{equation}
			for any \(h \in H, f \in \Hom_k^r(X, M)\) and \(x \in X\). Then \(\Hom_k^r(X, M)\) is a partial \(H\)-module.
			\item Let \(\Hom_k^\ell(X, M) = \Hom_k(X, M)\) as vector spaces, and put \begin{equation}
				(h \cdot f)(x) = h_{(2)} \bullet f(S^{-1}(h_{(1)}) \btright x)
			\end{equation}
			for any \(h \in H, f \in \Hom_k^\ell(X, M)\) and \(x \in X\). Then \(\Hom_k^\ell(X, M)\) is a partial \(H\)-module.
		\end{enumerate}
	\end{lemma}
	\begin{proof}
		For any \(h \in H, f \in \Hom_k^r(X, M)\) and \(x \in X\), 
		\[(h_{(1)} \cdot (S(h_{(2)}) \cdot f))(x) = h_{(1)} \bullet (S(h_{(2)}) \bullet f(x)),\]
		so it is easy to see that \ref{PR1}, \ref{PR2} and \ref{PR3} are satisfied. On the other hand, if \(f \in \Hom_k^\ell(X, M)\), then
		\[(S(h_{(1)}) \cdot (h_{(2)} \cdot f))(x) = S(h_{(1)}) \bullet (h_{(2)} \bullet f(x)),\]
		so \ref{PR1}, \ref{PR4} and \ref{PR5} hold.
	\end{proof}

	\begin{proposition}
		For any global \(H\)-module \(X,\)   we have the following adjunctions:
		\begin{equation}\label{eq:adjunctionX-}
			\xymatrix{\HPMod \ar@/^/[rr]^-{X\otimes -}_-{\perp} & & \HPMod \ar@/^/[ll]^-{\Hom_k^{\ell} (X, -)} } 
		\end{equation}
		\begin{equation}\label{eq:adjunction-X}
			\xymatrix{\HPMod \ar@/^/[rr]^-{-\otimes X}_-{\perp} & & \HPMod \ar@/^/[ll]^-{\Hom_k^r (X, -)} } 
		\end{equation}
	\end{proposition}
	\begin{proof}
		Let us show that the unit and counit of \eqref{eq:adjunctionX-},
		\begin{align*}
			\eta_M &: M \to \Hom_k^\ell(X, X \otimes M), \quad m \mapsto (x \mapsto (x \otimes m)), \\
			\epsilon_M &: X \otimes \Hom_k^\ell(X, M) \to M, \quad x \otimes f \mapsto f(x),
		\end{align*}
		are morphisms of partial modules. We have indeed
		\begin{align*}
			(h \cdot \eta_M(m))(x) &= h_{(2)} \bullet \eta_M(m)(S^{-1}(h_{(1)}) \btright x) = h_{(2)} \bullet (S^{-1}(h_{(1)}) \btright x \otimes m) \\
			&= h_{(2)}S^{-1}(h_{(1)}) \btright x \otimes h_{(3)} \bullet m = x \otimes h \bullet m; \\
			\epsilon_M(h \bullet (x \otimes f)) &= \epsilon_M(h_{(1)} \btright x \otimes h_{(2)} \cdot f) = (h_{(2)} \cdot f)(h_{(1)} \btright x) \\
			&= h_{(3)} \bullet f(S^{-1}(h_{(2)}) h_{(1)} \btright x) = h \bullet f(x).
		\end{align*}
		The proof for \eqref{eq:adjunction-X} is similar. 
	\end{proof}
	
	We turn now to the functors \(- \otimes M\) \eqref{eq:-M} and \(M \otimes -\) \eqref{eq:M-} for a partial \(H\)-module \(M\). 
	As explained in \cite[Section 2]{JKactegories}, if \((\Mm, \triangleright)\) is a left \(\Cc\)-actegory, then the right adjoint \(R_M\) to the functor \(- \triangleright M : \Cc \to \Mm\) makes \(\Mm\) a \(\Cc\)-enriched category, by defining the Hom-object from \(M\) to \(N\) as \(R_M(N)\).\footnote{In \cite[Section 7.9]{EGNObook}, a Hom-object arising from such an enrichment is called an \textit{internal Hom}. We prefer to avoid this terminology, since the Hom-object, being a global module, is not internal to the category of partial modules.}
	
	Notice that finding a right adjoint to our functor \(- \otimes M\) \eqref{eq:-M} amounts to constructing an object \([M, N]\) for any partial module \(N\) such that for all global modules \(X\)
	\[\Hom_{\HPMod}(X \otimes M, N) \cong {_H}\Hom(X, [M, N]).\]
	This can be done using the classical Hom-tensor adjunction. 
	The details are worked out in the next proposition. 
	
	\begin{proposition}
		Let \(M, N\) be partial \(H\)-modules and denote
		\begin{align*}
			[M, N] &= \Hom_{\HPMod}(H \otimes M, N), \\
			\{M, N\} &= \Hom_{\HPMod}(M \otimes H, N).
		\end{align*}
		We have the following adjunctions:
		\begin{equation}\label{eq:adjunction-M}
			\xymatrix{\HMod \ar@/^/[rr]^-{-\otimes M}_-{\perp} & & \HPMod \ar@/^/[ll]^-{[M, -] } }
		\end{equation}
		\begin{equation}\label{eq:adjunctionM-}
			\xymatrix{\HMod \ar@/^/[rr]^-{M\otimes -}_-{\perp} & & \HPMod \ar@/^/[ll]^-{\{M, -\}} } 
		\end{equation}
	\end{proposition}
	\begin{proof}
		We start with \eqref{eq:adjunctionM-}. For a partial module \(M\), consider the \((H_{par}, H)\)-bimodule \(M \otimes H,\) where the module structures are given by 
		\[h \bullet (m \otimes k) \btleft \ell = h_{(1)} \bullet m \otimes h_{(2)} k \ell\]
		for \(h, k, \ell \in H\) and \(m \in M\). Then \(M \otimes X \cong (M \otimes H) \otimes_H X\) for any global module \(X\) via the map \(m \otimes x \mapsto (m \otimes 1_H) \otimes_H x\). Now 
		\[{_{H_{par}}}\Hom(M \otimes X, N) \cong {_{H_{par}}}\Hom((M \otimes H) \otimes_H X, N) \cong {_H}\Hom(X, {_{H_{par}}}\Hom(M \otimes H, N)),\]
		where the second isomorphism is the tensor-hom adjunction. The \(H\)-module structure on \({_{H_{par}}}\Hom(M \otimes H, N) \eqqcolon \{M, N\}\) is induced by the right \(H\)-module structure on \(M \otimes H,\) i.\,e.\ 
		\begin{equation}
			\label{eq:Hmod_righthom}
			(h' \rightharpoonup f)(m \otimes h) = f(m \otimes hh')
		\end{equation}
		for all \(f \in \{M, N\}\) and \(h, h' \in H\). The natural isomorphisms
		\begin{equation}
			\label{eq:correspondence_righthom}
			\psi_X : \Hom_{\HPMod}(M \otimes X, N) \to \Hom_{\HMod}(X, \{M, N\})
		\end{equation}
		are given by \(\psi_X(f)(x)(m \otimes h) = f(m \otimes h \btright x)\) with inverse \(\psi^{-1}(g)(m \otimes x) = g(x)(m \otimes 1_H)\). The adjunction \eqref{eq:adjunction-M} is obtained in a similar way, by considering the \((H_{par}, H)\)-bimodule \(H \otimes M\) with module structures 
		\[h \bullet (k \otimes m) \btleft \ell = h_{(1)}k \ell \otimes h_{(2)} \bullet m.\] Then
		\(X \otimes M \cong (H \otimes M) \otimes_H X\) via the map \(x \otimes m \mapsto (1_H \otimes m) \otimes_H x\) and
		\[{_{H_{par}}}\Hom(X \otimes M, N) \cong {_{H_{par}}}\Hom((H \otimes M) \otimes_H X, N) \cong {_H}\Hom(X, {_{H_{par}}}\Hom(H \otimes M, N)). \qedhere\]
	\end{proof}

	\begin{theorem}\label{enrichment}
		\begin{enumerate}[(i)]
			\item The category of partial modules \(\HPMod\) is enriched over \(\HMod\) if the Hom-object between partial modules \(M\) and \(N\) is taken to be
			\([M, N]\).
			\item Let \(\HMod^{\mathrm{rev}}\) the monoidal category of \(H\)-modules with reversed tensor product. Then \(\HPMod\) is enriched over \(\HMod^{\mathrm{rev}}\), if the hom-object between partial modules \(M\) and \(N\) is taken to be
			\(\{M, N\}.\)
		\end{enumerate}
	\end{theorem}
	\begin{proof}

		We spell out the evaluation and composition morphisms of \(\{-, -\}\). 
		The evaluation map is found by taking \(X = \{M, N\}\) and \(g = \mathrm{id}_{\{M, N\}}\) in the correspondence \eqref{eq:correspondence_righthom}, giving
		\[e_{\{MN\}} : \{M, N\} \otimes^{\mathrm{rev}} M = M \otimes \{M, N\} \to N, \quad m \otimes \varphi \mapsto \varphi(m \otimes 1_H).\]
		Remark that \(\epsilon_N = e_{\{MN\}}\) is in fact the counit of the adjunction \(M \otimes - \dashv \{M, -\}\). 
		
		To find the composition law, consider first
		\[e_{\{NP\}}(e_{\{MN\}} \otimes \{N, P\}) : M \otimes \{M, N\} \otimes \{N, P\} \to P, \quad m \otimes \varphi \otimes \chi \mapsto \chi(\varphi(m \otimes 1_H) \otimes 1_H).\]
		Applying \(\psi_{\{M, N\} \otimes \{N, P\}}\) yields the composition morphism
		\begin{align}
			\circ_{\{MNP\}} : \{N, P\} \otimes^{\mathrm{rev}} \{M, N\} &= \{M, N\} \otimes \{N, P\} \to \{M, P\}, \label{eq:compositionright} \\ (\varphi \circ_{\{MNP\}} \chi)(m \otimes h) &= (h_{(2)} \rightharpoonup \chi)((h_{(1)} \rightharpoonup \varphi)(m \otimes 1_H) \otimes 1_H) = \chi(\varphi(m \otimes h_{(1)}) \otimes h_{(2)}). \nonumber
		\end{align}
		The unit of \(\{M, M\}\) is \(\psi_k(\Id_M)(1_k),\) which is the morphism \(m \otimes h \mapsto \epsilon(h) m\).
		The evaluation, composition and unit morphisms for \([-, -]\) are computed in a similar way.
	\end{proof}
	
	So far, we have discussed the ``local'' right adjoints to the action of \(\HMod\) on \(\HPMod\): the four functors \eqref{eq:X-}, \eqref{eq:-X}, \eqref{eq:-M} and \eqref{eq:M-} have adjoints for any global module \(X\) and partial module \(M\), and provide an enrichment of \(\HPMod\) over \(\HMod\). Recall that the left \(\HMod\)-actegory \eqref{eq:leftactegory} corresponds to a strong monoidal functor \(\HMod \to [\HPMod, \HPMod]\). As explained in \cite{JKactegories}, it is interesting to know whether this ``global'' action functor has a right adjoint, because it would automatically provide a monoidal adjunction between \(\HMod\) and \([\HPMod, \HPMod]\), which on its turn gives an adjunction between \(H\)-module algebras and monads on \(\HPMod\). Since our acting category \(\HMod\) is closed monoidal, \cite[Theorem 3.1]{JKactegories} can be applied; it states that for the action functor \(\HMod \to [\HPMod, \HPMod]\) to have a right adjoint, it suffices that each of the functors \eqref{eq:X-}-\eqref{eq:M-} has a right adjoint. 
	
	\begin{corollary}
		The strong monoidal functors 
		\begin{align}
			\HMod &\to [\HPMod, \HPMod], \quad X \mapsto (X \otimes -), \label{eq:action_left}\\
			\HMod &\to [\HPMod, \HPMod], \quad X \mapsto (- \otimes X). \label{eq:action_right}
		\end{align}
		have a right adjoint.
	\end{corollary}
	
	Let us describe the adjoint of \eqref{eq:action_right} explicitly. 
	For any partial module \(M\) and \(m \in M,\) there is a morphism
	\begin{equation}
		\label{eq:lambdam}
		\lambda_m : H_{par} \to M, \quad [h^1]\cdots[h^n] \mapsto h^1 \bullet (\cdots \bullet (h^n \bullet m)).
	\end{equation}
	Given an endofunctor \(F\) on \(\HPMod\) and \(x \in F(H_{par})\), we define
	\[x \bullet h = F(\lambda_{[h]})(x);\]
	this way \(F(H_{par})\) is a right partial \(H\)-module, so it is a right \(H_{par}\)-module. Since \(F(H_{par})\) also a left \(H_{par}\)-module and \(F(\lambda_{[h]})\) is a morphism of partial modules, \(F(H_{par})\) is an \(H_{par}\)-bimodule. 
	
	Localizing a natural transformation \(\alpha : - \otimes X \Rightarrow F\) at \(H_{par}\) gives an element of
	\[\HparHom(H_{par} \otimes X, F(H_{par})) \cong \HHom(X, \{H_{par}, F(H_{par})\}),\]
	where \(\{H_{par}, F(H_{par})\} = \HparHom(H_{par} \otimes H, F(H_{par}))\) as before.
	
	Conversely, any natural transformation \(\beta : - \otimes X \Rightarrow F\) is completely determined by \(\varphi = \psi_{X, H_{par}}(\beta_{H_{par}}) \in \HHom(X, \{H_{par}, F(H_{par})\})\) (with \(\psi\) as defined in \eqref{eq:correspondence_righthom}): it follows from the naturality that
	\[\beta_M : M \otimes X \to F(M), \quad m \otimes x \mapsto (F(\lambda_m)\circ \varphi(x))(1_{H_{par}} \otimes 1_H).\]
	This map is left \(H_{par}\)-linear if and only if
	\begin{align*}
		h \bullet (F(\lambda_m) \circ \varphi(x))(1_{H_{par}} \otimes 1_H) &= (F(\lambda_{h_{(1)}\bullet m}) \circ \varphi(h_{(2)} \triangleright x))(1_{H_{par}} \otimes 1_H) \\
		\Leftrightarrow \qquad F(\lambda_m)(h \bullet \varphi(x)(1_{H_{par}} \otimes 1_H)) &= (F(\lambda_m)\circ F(\lambda_{[h_{(1)}]}) \circ \varphi(x))(1_{H_{par}} \otimes h_{(2)}) \\
		\Leftrightarrow \qquad (F(\lambda_m) \circ \varphi(x))([h_{(1)}] \otimes h_{(2)}) &= (F(\lambda_m)\circ F(\lambda_{[h_{(1)}]}) \circ \varphi(x))(1_{H_{par}} \otimes h_{(2)}).
	\end{align*}
	This condition is satisfied if
	\begin{equation}
		\label{eqII:varphi_right_linear}
		\varphi(x)(z[h] \otimes h') = F(\lambda_{[h]})(\varphi(x)(z \otimes h')) = \varphi(x)(z \otimes h') \bullet h
	\end{equation}
	for all \(h, h' \in H\) and \(z \in H_{par}\) (and conversely \eqref{eqII:varphi_right_linear} holds for \(\varphi = \psi_{X, H_{par}}(\beta_{H_{par}})\) by naturality of \(\beta\)). This tells that \(\varphi\) is also right \(H_{par}\)-linear, where \(H_{par} \otimes H\) is considered as a free right \(H_{par}\)-module. In conclusion, 
	\[\mathsf{Nat}(- \otimes X, F) \cong \HHom(X, \HparHom_{H_{par}}(H_{par} \otimes H, F(H_{par}))),\]
	and the right adjoint to \eqref{eq:action_right} is
	\[F \mapsto Z_F = \HparHom_{H_{par}}(H_{par} \otimes H, F(H_{par})).\]
	We remark that \(Z_F\) is an \(H\)-submodule of \(\{H_{par}, F(H_{par})\}\) and that it is isomorphic to the bimodule center of the \(H_{par}\)-bimodule \(\Hom_k(H, F(H_{par}))\), where the left and right partial action of \(H\) on \(f \in \Hom_k(H, F(H_{par}))\) are given by
	\[(h \cdot f \cdot h')(x) = h_{(1)} \bullet f(S(h_{(2)})x) \bullet h'.\]

	\section{Globalization revisited} 
	
	\subsection{Dilation of a partial module}
	
	The construction of the standard dilation (\cref{th:standard_dilation}) gives a functor \(D : \HPMod \to \HMod\), which is intimately connected with the enriched structure of \(\HPMod\) via \(\{-, -\}\). Namely, we will establish a natural transformation \(D \Rightarrow \{A_{par}, -\}\) and show that it is an isomorphism if \(H\) is a pointed Hopf algebra with finitely many grouplikes.
	
	Recall that \(A_{par}\) is a partial \(H\)-module algebra by means of the partial action described in Section~\ref{se:partialprelim}. Hence \(A_{par} \otimes H\) is a partial \(H\)-module, where the module structure is given by the diagonal action:
	\[h \bullet (a \otimes k) = h_{(1)} \bullet a \otimes h_{(2)}k = [h_{(1)}] a [S(h_{(2)})] \otimes h_{(3)} k\]
	for all \(a \in A_{par}\) and \(h, k \in H\). Recall furthermore that there is an algebra isomorphism \(\sigma\) between the partial smash product \(A_{par} \underline{\#} H\) and the partial Hopf algebra \(H_{par}\) (see \eqref{eq:sigma} in Section~\ref{se:partialprelim}).
	
	\begin{lemma}
		\label{leII:projtoHpar}
		Let \(\pi\) be the canonical projection \[A_{par} \otimes H \to A_{par} \underline{\#} H,\ a \otimes h \mapsto a \# h = a(h_{(1)} \bullet 1_{A_{par}}) \otimes h_{(2)}\] and let \(\sigma\) be the isomorphism \(A_{par} \underline{\#} H \cong H_{par}\) , sending \(a \# h \in A_{par} \underline{\#} H\) to \(a[h]\). Then \(\tau \coloneqq \sigma \circ \pi\) is a morphism of partial modules.
	\end{lemma}
	\begin{proof}
		For any \(h \in H\), we have
		\begin{equation}
			\label{eq:triple_bracket}
			[h_{(1)}][S(h_{(2)})][h_{(3)}] \overset{\ref{PR3}}{=} [h_{(1)}][S(h_{(2)})h_{(3)}] \overset{\ref{PR1}}{=} [h].
		\end{equation}
		For any \(a \in A_{par}\) and \(l, h \in H,\) we calculate
		\begin{align*}
			\tau(l \bullet (a \otimes h)) \overset{\phantom{\ref{PR5}}}&{=} \tau([l_{(1)}]a[S(l_{(2)})] \otimes l_{(3)}h) = [l_{(1)}]a[S(l_{(2)})][l_{(3)}h] \\
			\overset{\ref{PR5}}&{=} [l_{(1)}]a[S(l_{(2)})][l_{(3)}][h] \overset{(*)}{=} [l_{(1)}][S(l_{(2)})][l_{(3)}]a[h] \\
			\overset{\eqref{eq:triple_bracket}}&{=} [l] a [h] = [l] \tau(a \otimes h).
		\end{align*}
		In step \((*)\) we used that by \cite[Lemma 4.7]{ABVparreps}, \(\tilde{\varepsilon}_k = [S(k_{(1)})][k_{(2)}]\) commutes with any element of \(A_{par}\), for any \(k \in H\).
	\end{proof}
	
	The composition of the inverse of \(\sigma\) with the inclusion of \(A_{par} \underline{\#} H\) in \(A_{par} \otimes H\) provides a section to \(\tau\) because \(a(h_{(1)} \bullet 1_{A_{par}}) \# h_{(2)} = a \# h\) by \eqref{eq:smash_rewriting}.
	Let us denote it by \(\bar{\tau} : H_{par} \to A_{par} \otimes H\).
	
	Let \(M\) be a partial module and consider the global module
	\[\{A_{par}, M\} = \Hom_{\HPMod}(A_{par} \otimes H, M).\]
	We will show that \(\{A_{par}, M\}\) is a dilation of \(M\) when equipped with a suitable linear projection. For any \(m \in M,\) there is the morphism of partial modules \(\lambda_m\) \eqref{eq:lambdam}. This defines an isomorphism \(M \cong \Hom_{\HPMod}(H_{par}, M)\), whose inverse is given by \(f \mapsto f(1_{H_{par}})\). Let \(\theta\) be the map \(M \cong \Hom_{\HPMod}(H_{par}, M) \to \{A_{par}, M\}\) induced by \(\tau\), i.\,e.\ 
	\begin{equation}
		\label{eq:theta}
		\theta : M \to \{A_{par}, M\}, \quad \theta(m)(a \otimes h) = a \cdot (h \bullet m),
	\end{equation}
	where \(\cdot\) denotes the left action \(A_{par}\) on \(M\) induced by the partial module structure (see \eqref{eq:leftAaction}). There is also a natural linear map in the other direction, induced by \(\bar{\tau}\),
	\begin{equation}
		\kappa : \{A_{par}, M\} \to M, \quad f \mapsto f(1_{A_{par}} \otimes 1_H).
	\end{equation}
	It is clear that \(\kappa \circ \theta = \Id_M\). 
	
	Let \(X\) be a global \(H\)-module. For any \(b \in A_{par}, h \in H\) and \(x \in X,\)
	\begin{equation*}
		\varepsilon_h \cdot (b \otimes x) = h_{(1)} \bullet (S(h_{(2)}) \bullet (b \otimes x)) = h_{(1)} \bullet (S(h_{(4)}) \bullet b) \otimes h_{(2)}S(h_{(3)}) \btright x = \varepsilon_h b \otimes x,
	\end{equation*}
	so, since \(A_{par}\) is generated by the elements \(\varepsilon_h\),
	\begin{equation}
		\label{eq:AMX}
		a \cdot (b \otimes x) = ab \otimes x
	\end{equation}
	for all \(a, b \in A_{par}\) and \(x \in X\), i.\,e.\ the left action of \(A_{par}\) on \(A_{par} \otimes X\) (and in particular on \(A_{par} \otimes H\)) only concerns the first tensorand.  
	
	We define
	\begin{align}
		\label{eq:calT}
		\Tt = \theta \circ \kappa :\ & \{A_{par}, M\} \to \{A_{par}, M\}, \\
		&\Tt(f)(a \otimes h) = a \cdot (h \bullet f(1_{A_{par}} \otimes 1_H)) \overset{\eqref{eq:AMX}}{=} f(a \varepsilon_{h_{(1)}} \otimes h_{(2)}). \nonumber
	\end{align}

	\begin{proposition}
		\label{propII:Apar_min_dilation}
		The triple \((\{A_{par}, M\}, \Tt, \theta)\) is a minimal dilation of \(M\).
	\end{proposition} 
	\begin{proof}
		Since \(\Tt = \theta \circ \kappa\) by definition and \(\kappa \circ \theta = \Id_M\), the linear map \(\Tt\) is a projection. 
		Let us show that \(\Tt\) obeys the \(c\)-condition \eqref{eq:ccondition}. Recall that the left \(H\)-module structure on \(\{A_{par}, M\}\) is given by \eqref{eq:Hmod_righthom}, i.\,e.\
		\((h \rightharpoonup f)(a \otimes k) = f(a \otimes kh).\) 
		For any \(h, k \in H\) we have on the one hand
		\begin{align*}
			\mathcal{T} (\mathcal{T}_h (f))(a\otimes k) &= \mathcal{T}_h(f) (a\varepsilon_{k_{(1)}} \otimes k_{(2)}) = (h_{(1)} \rightharpoonup \mathcal{T} (S(h_{(2)}) \rightharpoonup f)) (a\varepsilon_{k_{(1)}} \otimes k_{(2)}) \\
			&= \mathcal{T} (S(h_{(2)}) \rightharpoonup f) (a\varepsilon_{k_{(1)}} \otimes k_{(2)}h_{(1)}) = (S(h_{(3)}) \rightharpoonup f) (a\varepsilon_{k_{(1)}} \varepsilon_{k_{(2)}h_{(1)}} \otimes k_{(3)}h_{(2)}) \\
			&= f (a\varepsilon_{k_{(1)}} \varepsilon_{k_{(2)}h_{(1)}} \otimes k_{(3)}h_{(2)}S(h_{(3)})) = f (a\varepsilon_{k_{(1)}} \varepsilon_{k_{(2)}h} \otimes k_{(3)}) .
		\end{align*}
		On the other hand, 
		\begin{align*}
			\mathcal{T}_h (\mathcal{T} (f))(a\otimes k) &= (h_{(1)} \rightharpoonup \mathcal{T} (S(h_{(2)}) \rightharpoonup \mathcal{T}(f))) (a\otimes k) 
			= \mathcal{T} (S(h_{(2)}) \rightharpoonup \mathcal{T}(f)) (a\otimes kh_{(1)}) \\
			&= (S(h_{(3)}) \rightharpoonup \mathcal{T}(f)) (a \varepsilon_{k_{(1)}h_{(1)}}\otimes k_{(2)}h_{(2)}) 
			= \mathcal{T}(f) (a \varepsilon_{k_{(1)}h_{(1)}}\otimes k_{(2)}h_{(2)}S(h_{(3)})) \\
			&= \mathcal{T}(f) (a \varepsilon_{k_{(1)}h}
			\otimes k_{(2)})
			= f(a \varepsilon_{k_{(1)}h} \varepsilon_{k_{(2)}}
			\otimes k_{(3)}).
		\end{align*}
		Since $\varepsilon_{k_{(1)}h} \varepsilon_{k_{(2)}}= [k_{(1)}] \varepsilon_h [S(k_{(2)})] = \varepsilon_{k_{(1)}} \varepsilon_{k_{(2)}h}$ in $A_{par}$ by \cite[Lemma 4.7]{ABVparreps}, the above computation shows that \(\Tt \circ \Tt_h = \Tt_h \circ \Tt\) for all \(h \in H\). The image of \(\Tt\) is now a partial module via \(h \bullet f = \Tt(h \rightharpoonup f)\) \eqref{eq:TM} for all \(h \in H\) and \(f \in \Tt(\{A_{par}, M\})\).
		
		Since \(\Tt \circ \theta = \theta \circ \kappa \circ \theta = \theta\), the images of \(\theta\) and \(\Tt\) coincide. 
		It remains to show that \(\theta\) corestricts to an isomorphism of partial modules between \(M\) and \(\Tt(\{A_{par}, M\})\).
		For $m\in M$, $a\in A_{par}$ and $h,k\in H$, we have
		\[
		\theta (h\bullet m)(a\otimes k)=a\cdot (k\bullet (h \bullet m)) ,
		\]
		while
		\begin{align*}
			(h\bullet \theta (m))(a\otimes k) &= \mathcal{T}(h \rightharpoonup \theta (m))(a\otimes k) 
			= (h\rightharpoonup \theta (m)) (a\varepsilon_{k_{(1)}}\otimes k_{(2)}) \\
			&= \theta (m) (a\varepsilon_{k_{(1)}}\otimes k_{(2)}h) 
			= a\varepsilon_{k_{(1)}} \cdot (k_{(2)}h \bullet m)\\
			&= a\cdot (k_{(1)} \bullet (S(k_{(2)}) \bullet  (k_{(3)}h \bullet m))) \\
			\overset{\ref{PR5}}&{=} a\cdot (k_{(1)} \bullet (S(k_{(2)}) \bullet  (k_{(3)}\bullet (h \bullet m))))
			\overset{\eqref{eq:triple_bracket}}{=} a\cdot (k\bullet (h\bullet m)).
		\end{align*}
		We conclude that \((\{A_{par}, M\}, \mathcal{T}, \theta)\) is a dilation. To see that it is minimal, take \(f \in \{A_{par}, M\}\) such that \(\mathcal{T}(l \rightharpoonup f) = 0\) for all \(l \in H\). Then for all \(a \in A_{par}\) and \(h, l \in H\), we have \\ \(f(a \varepsilon_{h_{(1)}} \otimes h_{(2)} l) = 0\), so \(f\) is clearly zero. 
		This shows that no non trivial \(H\)-submodule of \(\{A_{par}, M\}\) is annihilated by \(\mathcal{T}\).
	\end{proof}
	
	\begin{proposition}
		\label{cor:Xi}
		The assignment
		\begin{align}
			\label{eqII:Xi}
			\Xi_M :\ & \overline{M} \to \{A_{par}, M\},\\ &\Xi_M\left(\sum_i h_i \rightharpoonup  \varphi(m_i)\right)(a \otimes k) =  a \cdot \sum_i (kh_i \bullet m_i) \nonumber
		\end{align}
		for any partial module \(M\) defines a natural transformation \(\Xi : D \Rightarrow \{A_{par}, -\}\) with injective components.
	\end{proposition}
	\begin{proof}
		By the universal property of the standard dilation (\cref{prop:universal_minimal}), there is a unique injective \(H\)-linear map \(\Xi_M : \overline{M} \to \{A_{par}, M\}\) such that \(\Tt \circ \Xi_M = \Xi_M \circ \overline{T}\) and \(\Xi_M \circ \varphi = \theta\). It maps \(x = \sum_i h_i \rightharpoonup  \varphi(m_i) \in \overline{M}\) to \(\sum_i h_i \rightharpoonup \theta(m_i)\), i.\,e\
		\[\Xi_M(x)(a \otimes k) = \sum_i \theta(m_i)(a \otimes kh_i) = a \cdot \sum_i (kh_i \bullet m_i).\]
		Let us check that it is natural in \(M\). Let \(f\) be a morphism of partial $H$-modules $M\rightarrow N$. Then for elements $h_i, k \in H$, $m_i \in M$ and $a\in A_{par}$, we have
		\begin{align*}
			\Xi_N \left( \sum_i h_i \rightharpoonup \varphi_N (f(m_i)) \right) (a\otimes k) &= a\cdot \sum_i kh_i \bullet f(m_i) 
			= f\left( a\cdot \sum_i kh_i \bullet m_i \right) \\
			&= f \left( \Xi_M \left( \sum_i h_i \rightharpoonup \varphi_M (m_i )\right) (a \otimes k)\right). \qedhere
		\end{align*}
	\end{proof}

	\noindent In fact, an element \(f \in \{A_{par}, M\}\) is completely determined by the linear map \(h \mapsto f(1_{A_{par}} \otimes h)\), essentially because \(A_{par}\) is a cyclic partial \(H\)-module. This allows us to view \(\{A_{par}, M\}\) as a submodule of \(\Hom_k(H, M)\), consisting of those maps that are \textit{partially \(H\)-linear}. We introduce this concept below.
	
	\begin{definition}
		Let \(X\) be a global \(H\)-module and \(M\) a partial \(H\)-module. A linear map \(g : X \to M\) is said to be \textit{partially \(H\)-linear} if \begin{equation}
			\label{eq:Psibar_image}
			h_{(1)} \bullet g(S(h_{(2)})\btright x) = h_{(1)} \bullet (S(h_{(2)}) \bullet g(x))
		\end{equation}
		for all \(h \in H\) and \(x \in X\). We denote the set of partially \(H\)-linear maps from \(X\) to \(M\) by \({_{(H)}\Hom(X,M)}\).
	\end{definition}
	In view of \eqref{eq:leftAaction}, \eqref{eq:Psibar_image} is equivalent to 
	\begin{equation}
		\label{eq:Psibar_image3}
		h \bullet g(x) = \varepsilon_{h_{(1)}} \cdot g(h_{(2)}\btright x)
	\end{equation}
	for all \(h \in H\) and \(x \in X\).

	\begin{proposition}
		\label{prop:Psibar} Let \(X\) be a global \(H\)-module and \(M\) a partial \(H\)-module.
		The linear map 
		\begin{equation}
			\label{eq:Psibar}
			\Psi :\Hom_{\HPMod}(A_{par} \otimes X, M) \rightarrow \Hom_k (X, M),\quad \Psi (f)(x)=f(1_{A_{par}}\otimes x)
		\end{equation}
		is injective and its image is exactly \({_{(H)}\Hom(X,M)}\).
	\end{proposition}
	\begin{proof}
		Take \(f \in \Hom_{\HPMod}(A_{par} \otimes X, M)\). Then \(f\) is in particular left \(A_{par}\)-linear. Hence for any \(a \in A_{par}\) and \(x \in X\),
		\[f(a \otimes x) \overset{\eqref{eq:AMX}}{=} f(a \cdot (1_{A_{par}} \otimes x)) = a \cdot f(1_{A_{par}} \otimes x) = a \cdot \Psi(f)(x).\]
		This shows that \(f\) is completely determined by \(\Psi(f)\).
		Moreover,
		\begin{align*}
			h_{(1)} \bullet \Psi(f)(S(h_{(2)})\btright x) &= h_{(1)} \bullet f(1_{A_{par}} \otimes S(h_{(2)}) \btright x) 
			= f(h_{(1)} \bullet 1_{A_{par}} \otimes h_{(2)}S(h_{(3)})\btright x) \\
			&= f(\varepsilon_h \otimes x) 
			= \varepsilon_h \cdot f(1_{A_{par}} \otimes x) = h_{(1)} \bullet (S(h_{(2)}) \bullet \Psi(f)(x)), 
		\end{align*}
		so \(\Psi(f) \in {_{(H)}\Hom(X,M)}\).
		Suppose now that \(g : H \to M\) is partially \(H\)-linear, and define
		\begin{equation}
			\label{eq:ffromg}
			f : A_{par} \otimes X \to M, \quad a \otimes x \mapsto a \cdot g(x).
		\end{equation}
		Then for any \(a = \varepsilon_{k^1} \cdots \varepsilon_{k^{n}}\) and \(h\in H\) we have
		\begin{align*}
			f(h_{(1)} \bullet a \otimes h_{(2)}\btright x) &= f(\varepsilon_{h_{(1)}k^1} \cdots \varepsilon_{h_{(n)}k^n} \varepsilon_{h_{(n + 1)}} \otimes h_{(n + 2)}\btright x) \\
			&= \varepsilon_{h_{(1)}k^1} \cdots \varepsilon_{h_{(n)}k^n} \varepsilon_{h_{(n + 1)}} \cdot g(h_{(n + 2)} \btright x) \\
			\overset{\eqref{eq:Psibar_image3}}&{=} \varepsilon_{h_{(1)}k^1} \cdots \varepsilon_{h_{(n)}k^n} \cdot (h_{(n + 1)} \bullet g(x)) \\
			\overset{(*)}&{=} h \bullet( \varepsilon_{k^1} \cdots \varepsilon_{k^{n}} \cdot g(x)) = h \bullet f(a \otimes x).
		\end{align*}
		In equality \((*)\), we used the fact that in \(H_{par},\)
		\(\varepsilon_{h_{(1)}k}[h_{(2)}] = [h]\varepsilon_k\)
		for all \(h, k \in H\) (see \cite[Lemma 4.7(b)]{ABVparreps}). Hence \(f \in \Hom_{\HPMod}(A_{par} \otimes X, M),\) and \(\Psi(f) = g\).
	\end{proof}

	If \(X\) is an \(H\)-bimodule, then \(\Hom_{\HPMod}(A_{par} \otimes X, M)\) is a left \(H\)-module by means of the action
	\[(h \rightharpoonup f)(a \otimes x) = f(a \otimes x \btleft h),\]
	mimicking \eqref{eq:Hmod_righthom}. The map \(\Psi\) \eqref{eq:Psibar} is then clearly \(H\)-linear, so \({_{(H)}\Hom(X,M)}\) is an \(H\)-submodule of \(\Hom_k(X, M)\) and \(\Hom_{\HPMod}(A_{par} \otimes X, M) \cong {_{(H)}\Hom(X,M)}\) as \(H\)-modules. Specializing to \(X = H\) gives the following corollary of \cref{prop:Psibar}.
	
	\begin{corollary} 
		\label{cor:iso_Psi}
		Let \(M\) be a partial \(H\)-module. Then \(\{A_{par}, M\} \cong {_{(H)}\Hom(H,M)}\) as \(H\)-modules. 
	\end{corollary}

	\begin{remark}
		\label{rem:2nd_condition}
		A partially \(H\)-linear map \(g : X \to M\) also satisfies
		\begin{equation}
			\label{eq:Psibar_image2}
			S(h_{(1)}) \bullet g(h_{(2)}\btright x) = S(h_{(1)}) \bullet (h_{(2)} \bullet g(x))
		\end{equation}
		for all \(h \in H\) and \(x \in X\). Indeed, let \(f \in \Hom_{\HPMod}(A_{par} \otimes X, M)\) be as defined in \eqref{eq:ffromg}. Then 
		\begin{align*}
			S(h_{(1)}) \bullet (h_{(2)} \bullet g(x)) &= S(h_{(1)}) \bullet (h_{(2)} \bullet f(1_{A_{par}} \otimes x)) 
			= S(h_{(1)}) \bullet f(\varepsilon_{h_{(2)}} \otimes h_{(3)}\btright x) \\
			&= f(S(h_{(2)}) \bullet \varepsilon_{h_{(3)}} \otimes S(h_{(1)})h_{(4)}\btright x) \\
			&= f(\varepsilon_{S(h_{(3)}) h_{(4)}} \varepsilon_{S(h_{(2)})} \otimes S(h_{(1)})h_{(5)}\btright x) \\
			&= f(\varepsilon_{S(h_{(2)})} \otimes S(h_{(1)})h_{(3)}\btright x) 
			= S(h_{(1)}) \bullet f(1_{A_{par}} \otimes h_{(2)} \btright x) \\
			&= S(h_{(1)}) \bullet g(h_{(2)}\btright x).
		\end{align*}
	\end{remark}

	\subsection{The pointed case}
	
	In this section, we show that whenever \(H\) is a pointed Hopf algebra with finitely many grouplikes over a field of characteristic zero, then \(\Xi\) is a natural isomorphism. This result applies in particular to all finite group algebras and to the Sweedler Hopf algebra \(H_4\).
	The following proposition provides a sufficient condition for \(\Xi\) to be an isomorphism.
	
	\begin{proposition}
		\label{le:Xi_iso_b}
		Let \(t \in H\) be a left integral and suppose that there exists an element \(b \in A_{par}\) such that \(t \bullet b = 1_{A_{par}}\). Then the morphism \(\Xi_M \) from \cref{cor:Xi} is an isomorphism for any partial module \(M\). 
		
		In particular, if \(\varepsilon_t\) is invertible in \(A_{par},\) then \(t \bullet \varepsilon_t^{-1} = 1_{A_{par}}\), so it suffices to take \(b = \varepsilon_t^{-1}\).
	\end{proposition}
	\begin{proof}
		In view of \cref{cor:Xi}, we only need to show the surjectivity. Take \(f \in \{A_{par}, M\}\). Consider
		\[y = t_{(1)} \rightharpoonup \varphi(f(b \otimes S(t_{(2)}))) \in \overline{M}.\]
		Then, for all $h\in H$
		\begin{align*}
			\Psi\Xi_M (y)(h) &= \Xi_M (y)(1_{A_{par}} \otimes h) = ht_{(1)} \bullet f(b\otimes S(t_{(2)})) \\
			&= f( h_{(1)}t_{(1)} \bullet b \otimes h_{(2)}t_{(2)} S(t_{(3)})) = f( h_{(1)}t \bullet b \otimes h_{(2)}) \\
			&= f(t \bullet b \otimes h) = f(1_{A_{par}} \otimes h) = \Psi(f)(h).
		\end{align*}
		The injectivity of \(\Psi\) (\cref{prop:Psibar}) implies now that \(f = \Xi_M (y)\).
		
		Suppose now that \(\varepsilon_t\) is invertible in \(A_{par}\) (then also \(\epsilon(t)\) is invertible in \(k\), so \(H\) is necessarily semisimple). Then	
		\begin{align*}
			\varepsilon_t = t \bullet 1_{A_{par}} &= (t_{(1)} \bullet \varepsilon_t^{-1})(t_{(2)} \bullet \varepsilon_t) = (t_{(1)} \bullet \varepsilon_t^{-1}) \varepsilon_{t_{(2)}} \varepsilon_{t_{(3)}t} = (t_{(1)} \bullet \varepsilon_t^{-1}) \varepsilon_{t_{(2)}} \varepsilon_{t} \\ &= (t_{(1)} \bullet \varepsilon_t^{-1}) (t_{(2)} \bullet 1_{A_{par}}) \varepsilon_{t} 
			= (t \bullet \varepsilon_t^{-1})\varepsilon_{t} 
		\end{align*} 
		so that \(t \bullet \varepsilon_t^{-1} = 1_{A_{par}}\) because \(\varepsilon_t\) is invertible. Hence we can take \(b = \varepsilon_t^{-1}\) and conclude that in this case \(\Xi\) is an isomorphism.
	\end{proof}
	
	Next, we show that the sufficient condition obtained in \cref{le:Xi_iso_b} holds for finite group algebras over a field of characteristic zero.
	
	\begin{lemma}
		\label{le:b}
		Let \(G\) be a finite group of cardinality \(n\) and let \(k\) be a field of characteristic zero. Let \(t = \frac{1}{n} \sum_{g \in G} g \in kG\) the normalized integral. Denote by \(\mathcal{X}_i\) the set of subsets of \(G \setminus \{1_G\}\) of cardinality \(i\). Then
		\begin{equation}
			\label{eqII:b_group}
			\varepsilon_t^{-1} = n\sum_{i = 0}^{n - 1} \frac{(-1)^i}{i + 1} \sum_{X \in \mathcal{X}_i} \prod_{g \in X}\varepsilon_g = n \left(1_{A_{par}} - \frac{1}{2} \sum_{g \in G \setminus \{1\}} \varepsilon_g + \cdots\right) \in A_{par}(kG). 
		\end{equation}
	\end{lemma}
	\begin{proof}
		The elements \(\varepsilon_g \in A_{par}(kG)\) are commuting idempotents (see \cref{ex:group}), so the expression \eqref{eqII:b_group} is unambiguous. For any \(h \in G \setminus \{1_G\}\), denote by \(\mathcal{Y}_{h,i}\) the subset of \(\mathcal{X}_i\) consisting of sets that contain \(h\). Remark that \(\mathcal{Y}_{h, 0}\) and \(\mathcal{Y}_{h, n}\) are empty. Let \(b\) be the element \(n\sum_{i = 0}^{n - 1} \frac{(-1)^i}{i + 1} \sum_{X \in \mathcal{X}_i} \prod_{g \in X}\varepsilon_g\). Then
		\begingroup
		\allowdisplaybreaks
		\begin{align*}
			\varepsilon_t b = \frac{1}{n} \sum_{h \in G} \varepsilon_h b &= \sum_{i = 0}^{n - 1} \frac{(-1)^i}{i + 1} \sum_{X \in \mathcal{X}_i} \prod_{g \in X}\varepsilon_g + \sum_{h \in G \setminus \{1_G\}}  \sum_{i = 0}^{n - 1} \frac{(-1)^i}{i + 1} \sum_{X \in \mathcal{Y}_{h, i}} \prod_{g \in X}\varepsilon_g  \\
			&\quad + \sum_{h \in G \setminus \{1_G\}} \sum_{i = 0}^{n - 1} \frac{(-1)^i}{i + 1} \sum_{X \in \mathcal{Y}_{h, i + 1}} \prod_{g \in X}\varepsilon_g \\
			&\overset{(*)}{=} \sum_{i = 0}^{n - 1} \frac{(-1)^i}{i + 1} \sum_{X \in \mathcal{X}_i} \prod_{g \in X}\varepsilon_g + \sum_{i = 1}^{n - 1} \frac{i(-1)^i}{i + 1} \sum_{X \in \mathcal{X}_{i}} \prod_{g \in X}\varepsilon_g 
			+ \sum_{i = 0}^{n - 1} \frac{(i + 1)(-1)^{i}}{i + 1} \sum_{X \in \mathcal{X}_{i + 1}} \prod_{g \in X}\varepsilon_g \\
			&= \sum_{i = 0}^{n - 1} \frac{(-1)^i}{i + 1} \sum_{X \in \mathcal{X}_i} \prod_{g \in X}\varepsilon_g + \sum_{i = 1}^{n - 1} \frac{i(-1)^i}{i + 1} \sum_{X \in \mathcal{X}_{i}} \prod_{g \in X}\varepsilon_g 
			+ \sum_{i = 1}^{n - 1} (-1)^{i-1}\sum_{X \in \mathcal{X}_{i}} \prod_{g \in X}\varepsilon_g \\
			&= 1_{A_{par}} + \sum_{i = 1}^{n - 1} (-1)^i\left( \frac{1}{i + 1} + \frac{i}{i + 1} - 1\right)\sum_{X \in \mathcal{X}_{i}} \prod_{g \in X}\varepsilon_g \\
			&= 1_{A_{par}}.
		\end{align*}
		\endgroup
		Equality \((*)\) holds because in the second and third sum of the expression, each subset \(X\) of \(G \setminus \{1_G\}\) appears exactly \(|X|\) times. Since \(A_{par}(kG)\) is commutative, the above computation shows that \(b\) is the inverse of \(\varepsilon_t\).
	\end{proof}

	\begin{lemma}
		\label{le:restriction_coradical}
		Let \(H\) be a pointed Hopf algebra and take \(g : H \to M\) a partially linear map (i.\,e.\ satisfying \eqref{eq:Psibar_image}). Then \(g\) is completely determined by its restriction to the coradical \(H_0\).
	\end{lemma}
	\begin{proof}
		It suffices to show that \(g|_{H_0} = 0\) implies that \(g = 0\). So, suppose that \(g|_{H_0} = 0\). 
		Consider the coradical filtration \(\{H_n\}_{n \in \NN}\), where \(H_0\) is the coradical of \(H\) and 
		\[H_n = \Delta^{-1}(H \otimes H_0 + H_{n - 1} \otimes H).\] We will show by induction that \(g|_{H_n} = 0\) for all \(n \in \NN\), which will conclude the proof that \(g = 0\) because \(H = \bigcup_{n \in \NN} H_n\).
		
		By \cite[Equation 4.5 and Proposition 4.3.1]{radford}, \(H_n\) is generated as a vector space by elements \(x\) for which
		\[\Delta(x) = c \otimes x + x \otimes d + \sum_i y_i \otimes z_i,\]
		where \(c\) and \(d\) are grouplike, and \(\sum_i y_i \otimes z_i \in H_{n - 1} \otimes H_{n - 1}\). Then
		\[\Delta(xc^{-1}) = 1_H \otimes xc^{-1} + xc^{-1} \otimes dc^{-1} + \sum_i y_ic^{-1} \otimes  z_ic^{-1}.\]
		By \cref{rem:2nd_condition}, \(g\) also satisfies condition \eqref{eq:Psibar_image2}, which for \(h = xc^{-1}\) becomes
		\begin{gather*}
			g(xc^{-1}k) + S(xc^{-1}) \bullet g(dc^{-1}k) + \sum_i S(y_ic^{-1}) \bullet g(z_ic^{-1}k) \\ = xc^{-1} \bullet g(k) + S(xc^{-1}) \bullet (dc^{-1} \bullet g(k)) + \sum_i S(y_ic^{-1}) \bullet (z_ic^{-1} \bullet g(k)),
		\end{gather*}
		and by taking \(k = c\) we obtain
		\begin{align*}
			g(x) &= xc^{-1} \bullet g(c) + S(xc^{-1}) \bullet (dc^{-1} \bullet g(c)) + \sum_i S(y_ic^{-1}) \bullet (z_ic^{-1} \bullet g(c)) \\ &\quad - S(xc^{-1}) \bullet g(d) - \sum_i S(y_ic^{-1}) \bullet g(z_i).
		\end{align*}
		Since \(c, d, z_i \in H_{n - 1}\), this shows that \(g|_{H_{n - 1}} = 0 \Rightarrow g|_{H_{n}} = 0,\) and the proof by induction is finished.
	\end{proof}
	
	\begin{theorem}
		\label{th:Xi_iso_pointed}
		Let \(H\) be a pointed Hopf algebra with finite group of grouplikes \(G\) over a field of characteristic zero. Then \(\Xi\) is a natural isomorphism, so \(D \cong \{A_{par}, -\}\).
	\end{theorem}
	\begin{proof}
		Let \(t = \frac{1}{n}\sum_{g \in G} g\), which is a left integral in the coradical \(H_0 = kG\) of \(H\). Let \(b = \varepsilon_t^{-1} \in A_{par}(kG)\) be the element defined in \cref{le:b}. Since \(kG\) is a Hopf subalgebra of \(H\), there is an obvious morphism of partial $\Bbbk G$-module algebras
		\[\iota : A_{par}(kG) \to A_{par}(H), \quad \varepsilon_h \mapsto \varepsilon_h.\] Now \(t \bullet \iota(b) = \iota(t \bullet b) = \iota(1_{A_{par}(kG)}) = 1_{A_{par}(H)}\) by \cref{le:Xi_iso_b}. Let \(M\) be a partial \(H\)-module and take \(f \in \{A_{par}, M\}\). Put
		\[y = t_{(1)} \rightharpoonup \varphi(f(\iota(b) \otimes S(t_{(2)}))) \in \overline{M}.\]
		Then for all \(h \in H_0 = kG,\)
		\[\Psi\Xi_M(y)(h) = \Psi(f)(h)\]
		by the proof of \cref{le:Xi_iso_b}, and by \cref{le:restriction_coradical}, this implies that \(\Psi\Xi_M(y) = \Psi(f)\). Since \(\Psi\) is injective, \(f = \Xi_M(y)\). We conclude that \(\Xi_M\) is surjective, and by \cref{cor:Xi}, it is bijective. 
	\end{proof}
	
	\begin{corollary}\label{D_exato}
		Let \(H\) be a pointed Hopf algebra with finite group of grouplikes \(G\) over a field of characteristic zero. Then the dilation functor \(D\) is exact. Moreover, \(D\) has both a left and right adjoint, and 
		\[D \cong \{A_{par}, -\} \cong \overline{H_{par}} \otimes_{H_{par}} -.\]
	\end{corollary}
	\begin{proof}
		By \cref{th:Xi_iso_pointed}, \(D \cong \{A_{par}, -\}\), so the dilation functor is right adjoint to \(A_{par} \otimes - : \HMod \to \HPMod\). In particular, it preserves kernels, so by \cite[Proposition 5.2]{ABVdilations}, \(D\) is exact, has a right adjoint and is isomorphic to \(\overline{H_{par}} \otimes_{H_{par}} -\).
	\end{proof}

	\section{The Hopf algebroid \texorpdfstring{\(H_{glob}\)}{Hglob}} 
	For any partial module \(M\), \(\{A_{par}, M\}\) is not only a global \(H\)-module, but it is also a left module over the \(H\)-module algebra \(\{A_{par}, A_{par}\}\) using the composition \(\circ_{\{AAM\}}\) \eqref{eq:compositionright}: for any \(\chi \in \{A_{par}, A_{par}\}\) and \(\varphi \in \{A_{par}, M\}\),
	\[(\chi \circ_{\{AAM\}} \varphi)(a \otimes h) = \varphi(\chi(a \otimes h_{(1)}) \otimes h_{(2)}).\]
	The action of \(\{A_{par}, A_{par}\}\) on \(\{A_{par}, M\}\) is \(H\)-linear because \(\circ_{\{AAM\}}\) is a morphism in \(\HMod\). This means that \(\{A_{par}, M\}\) is an object in \({_{\{A_{par}, A_{par}\}}}(\HMod)\), the category of left \(\{A_{par}, A_{par}\}\)-modules in the monoidal category of \(H\)-modules. 
	If \(H\) is finite-dimensional and pointed, then \(D \cong \{A_{par}, -\}\) by \cref{th:Xi_iso_pointed}, so in that case the dilations of partial modules can be interpreted as objects in \({_{\{A_{par}, A_{par}\}}}(\HMod)\).
	
	\begin{lemma}
		\label{le:Apar_convolution}
		The algebra structure on \(\{A_{par}, A_{par}\}\) given by the composition coincides with the algebra structure defined by he following convolution product:
		\[
		(f\ast g)(a\otimes h) =f(a \otimes h_{(1)}) g(1_A \otimes h_{(2)}) .
		\]
	\end{lemma}
	\begin{proof}
		Consider $f,g \in \{ A_{par} , A_{par} \}$, $a\in A_{par}$ and $h\in H$, then
		\begin{equation*}
			(f*g)(a \otimes h) = f(a \otimes h_{(1)}) g(1_A \otimes h_{(2)})
			= g(f(a \otimes h_{(1)}) \otimes h_{(2)})
			= (f \circ_{\{AAA\}} g)(a \otimes h).
		\end{equation*}
		Here we used the fact that, since \(g\) is \(H_{par}\)-linear, it is in particular \(A_{par}\)-linear. Moreover, \(a' \cdot (1_{A_{par}} \otimes h) = a' \otimes h\) for any \(a' \in A_{par}\) and \(h \in H\) (see \eqref{eq:AMX}).
	\end{proof}
	
	\begin{definition}
		We put \(H_{glob} \coloneqq \{A_{par}, A_{par}\} \# H\). The multiplication on \(H_{glob}\) is given by
		\begin{equation}
			(\varphi \# h)(\psi \# l) = \varphi \circ_{\{AAA\}} (h_{(1)} \rightharpoonup \psi) \# h_{(2)}l.
		\end{equation}
	\end{definition}
	
	The category of left \(H_{glob}\)-modules is isomorphic to \({_{\{A_{par}, A_{par}\}}}(\HMod)\): given an \(H_{glob}\)-module \(Y\), it is an \(H\)-module by \(h \btright y \coloneqq (1_{\{A_{par}, A_{par}\}} \# h) \cdot y\), and 
	\[\varphi \triangleright y \coloneqq (\varphi \# 1_H) \cdot y\]
	defines an \(H\)-linear action of \(\{A_{par}, A_{par}\}\) on \(Y\). Conversely, an object \(Z\) in \({_{\{A_{par}, A_{par}\}}}(\HMod)\) is an \(H_{glob}\)-module via
	\[(\varphi \# h) \cdot z \coloneqq \varphi \triangleright (h \btright z).\]
	
	We will show that if \(H\) is finite-dimensional, then \(H_{glob}\) is a Hopf algebroid with base algebra \(\{A_{par}, A_{par}\}\), using \cite[Theorem 4.1]{BM}. To apply this theorem, we need to show that \(\{A_{par}, A_{par}\}\) is a braided commutative algebra in the category of Yetter-Drinfeld modules \(\HYDH\). First, we need the following lemma, which is a part of \cite[Lemma 7.9.4]{EGNObook}.
	
	\begin{lemma}
		\label{leII:Theta}
		For any finite-dimensional \(H\)-module \(X\) and partial modules \(M\) and \(N,\)
		\[\{M, N \otimes X\} \cong \{M, N\} \otimes X\]
		as left \(H\)-modules.
	\end{lemma}
	\begin{proof}
		We state here the isomorphism \(\Theta : \{M, N\} \otimes X \to \{M, N \otimes X\}\) but leave out the details. 
		Take \(z = \sum_i g_i \otimes x_i \in \{M, N\} \otimes X\) and define \(\Theta(z) : M \otimes H \to N \otimes X\) by
		\begin{equation} \label{eq:Theta} \Theta(z)(m \otimes h) = \sum_i g_i(m \otimes h_{(1)}) \otimes h_{(2)} \btright x_i. \end{equation}
		Then \(\Theta(z)\) is a morphism of partial \(H\)-modules, so \(\Theta(z) \in \{M, N \otimes X\}\). It is a direct check that \(\Theta\) is \(H\)-linear. The inverse of \(\Theta\) is constructed as follows. 
		Take \(f \in \{M, N \otimes X\}\) and let \(\{y_j \mid j \in J\}\) be a basis of \(X\). Denote the dual basis by \(\{y^*_j \mid j \in J\} \subseteq X^*\) and define for all \(j \in J\)
		\begin{equation*}
			g_j (m \otimes h) = (N \otimes y_j^* \leftharpoonup S^{-1}(h_{(2)})) f(m \otimes h_{(1)}) = f(m \otimes h_{(1)})^{[1]} y_j^*(S^{-1}(h_{(2)}) \btright f(m \otimes h_{(1)})^{[2]} ),
		\end{equation*}
		where we used the notation \(f(m \otimes h) = f(m \otimes h)^{[1]} \otimes f(m \otimes h)^{[2]} \in N \otimes X\).
		One can verify that \(g_j : M \otimes H \to N\) is a morphism of partial \(H\)-modules, so \(g_j \in \{M, N\}\).
		Now \(\Theta^{-1}(f) = \sum_{j \in J} g_j \otimes y_j\). 
	\end{proof}
	In particular, \(\{A_{par}, A_{par} \otimes H\} \cong \{A_{par}, A_{par}\} \otimes H\). 	
	To define an appropriate coaction on \(\{A_{par}, A_{par}\},\) we will make use of the following isomorphism \(A_{par} \otimes - \cong - \otimes A_{par}\) as functors \(\HMod \to \HPMod\).
	\begin{lemma}
		\label{le:AYD}
		Let \(X\) denote an \(H\)-module. There is a natural isomorphism of partial \(H\)-modules
		\begin{align*}
			\beta_X :\, & X \otimes A_{par} \to A_{par} \otimes X, \\
			&\beta_X(x \otimes \varepsilon_{h^1} \cdots \varepsilon_{h^n}) = \varepsilon_{h^1_{(2)}} \cdots \varepsilon_{h^n_{(2)}} \otimes h^n_{(3)} S^{-1}(h^n_{(1)})  \cdots h^1_{(3)} S^{-1}(h^1_{(1)}) \btright x.
		\end{align*}
	\end{lemma}
	\begin{proof}
		Using the universal property of \(A_{par}\) (see \cite[Theorem 4.12]{ABVparreps}), one shows that \(\beta_X\) is well-defined. Let us show that it is a morphism of partial modules. For \(k \in H,\) have
		
		\begin{align*}
			\beta_X(k \bullet (x \otimes \varepsilon_h)) &= \beta_X(k_{(1)} \btright x \otimes k_{(2)} \bullet \varepsilon_h) = \beta_X(k_{(1)} \btright x \otimes  \varepsilon_{k_{(2)} h} \varepsilon_{k_{(3)}}) \\
			&= \varepsilon_{k_{(3)}} \varepsilon_{k_{(6)} h_{(2)}} \otimes  k_{(7)} h_{(3)} S^{-1}(k_{(5)} h_{(1)}) k_{(4)} S^{-1}(k_{(2)}) k_{(1)} \btright x  \\
			&= \varepsilon_{k_{(1)}} \varepsilon_{k_{(2)} h_{(2)}} \otimes  k_{(3)} h_{(3)} S^{-1}( h_{(1)})  \btright x
			= k_{(1)} \bullet \varepsilon_{h_{(2)}} \otimes k_{(2)} h_{(3)} S^{-1}( h_{(1)})  \btright x \\
			&= k \bullet \beta_X(x \otimes \varepsilon_h).
		\end{align*}
		The inverse of \(\beta_X\) is given by
		\[A_{par} \otimes X \to X \otimes A_{par}, \quad \varepsilon_{h^1} \cdots \varepsilon_{h^n} \otimes x \mapsto h^1_{(1)} S(h^1_{(3)}) \cdots h^n_{(1)} S(h^n_{(3)}) \btright x \otimes \varepsilon_{h^1_{(2)}} \cdots \varepsilon_{h^n_{(2)}}. \qedhere\]
	\end{proof}
	
	Consider the linear map \(\rho_R : A_{par} \to A_{par} \otimes H\) defined as
	\begin{equation}
		\label{eqII:recall_rhoR}
		\rho_R(\varepsilon_{h^1} \cdots \varepsilon_{h^n})=\varepsilon_{h^1_{(2)}} \cdots \varepsilon_{h^n_{(2)}} \otimes h^n_{(3)} S^{-1}(h^n_{(1)}) \cdots h^1_{(3)} S^{-1}(h^1_{(1)}).
	\end{equation}
	Endowed with this coaction, \(A_{par}\) becomes a right \(H\)-comodule, and if we write \(\rho_R(a) = a^{(0)} \otimes a^{(1)}\) for \(a \in A_{par}\), then
	\[\beta_X(x \otimes a) = a^{(0)} \otimes a^{(1)} \btright x\]
	for any global module \(X\). 
	
	\begin{lemma}
		For any \(a, b \in A_{par}\),
		\begin{gather}
			\label{eqII:Apar_commutative2}
			b^{(0)}(b^{(1)} \bullet a) = ab, \\
			a^{(0)} (a^{(1)} \bullet (S(a^{(2)}) \bullet b)) = ab. \label{eqII:relation_ab}
		\end{gather}
	\end{lemma}
	\begin{proof}
		Take \(b = \varepsilon_{h^1} \cdots \varepsilon_{h^n} \in A_{par}\) and \(k \in H\). In the following computation, we apply the relation \(\varepsilon_{h_{(1)}} \varepsilon_{h_{(2)}k} = \varepsilon_{h_{(1)}k} \varepsilon_{h_{(2)}}\) in \(A_{par}\) (see \cite[\S 4.4]{ABVparreps}) repeatedly. We have
		\begin{align*}
			b^{(0)} (b^{(1)} \bullet \varepsilon_k) &= \varepsilon_{h^1_{(2)}} \cdots \varepsilon_{h^n_{(2)}} (h^n_{(3)} S^{-1}(h^n_{(1)}) \cdots h^1_{(3)} S^{-1}(h^1_{(1)}) \bullet \varepsilon_k) \\
			&= \varepsilon_{h^1_{(3)}} \cdots \varepsilon_{h^n_{(3)}} \varepsilon_{h^n_{(4)} S^{-1}(h^n_{(2)}) \cdots h^1_{(4)} S^{-1}(h^n_{(2)})} \varepsilon_{h^n_{(5)} S^{-1}(h^n_{(1)}) \cdots h^1_{(5)} S^{-1}(h^1_{(1)})k} \\
			&= \varepsilon_{h^1_{(3)}} \cdots \varepsilon_{h^{n - 1}_{(3)}}  \varepsilon_{h^{n - 1}_{(4)} S^{-1}(h^{n - 1}_{(2)}) \cdots h^1_{(4)} S^{-1}(h^n_{(2)})}
			\varepsilon_{h^n_{(2)}} \varepsilon_{h^n_{(3)} S^{-1}(h^n_{(1)}) h^{n - 1}_{(5)} S^{-1}(h^{n - 1}_{(1)}) \cdots h^1_{(5)} S^{-1}(h^1_{(1)})k} \\
			&= \cdots = \varepsilon_{h^1_{(2)}} \cdots \varepsilon_{h^n_{(2)}} \varepsilon_{h^n_{(3)} S^{-1}(h^n_{(1)}) \cdots h^1_{(3)} S^{-1}(h^1_{(1)})k} \\
			&= \varepsilon_{h^1_{(2)}} \cdots \varepsilon_{h^{n - 1}_{(2)}} \varepsilon_{h^{n - 1}_{(3)} S^{-1}(h^{n - 1}_{(1)}) \cdots h^1_{(3)} S^{-1}(h^1_{(1)})k}\varepsilon_{h^n} \\
			&= \cdots = \varepsilon_k \varepsilon_{h^1} \cdots \varepsilon_{h^n}.
		\end{align*}
		By linearity, \(\varepsilon_k b = b^{(0)} (b^{(1)} \bullet \varepsilon_k)\) for all \(b \in A_{par}\).
		Now, if \(a, a' \in A_{par}\) are such that \(ab = b^{(0)}(b^{(1)} \bullet a)\) and \(a'b = b^{(0)}(b^{(1)} \bullet a')\) for all \(b \in A_{par}\), then
		\begin{equation*}
			b^{(0)} (b^{(1)} \bullet (aa')) = b^{(0)}(b^{(1)} \bullet a)(b^{(2)} \bullet a')
			= a b^{(0)} (b^{(1)} \bullet a') = aa'b.
		\end{equation*}
		Since \(A_{par}\) is generated as an algebra by the elements \(\varepsilon_k\) for \(k \in H\), the two computations above show that \(ab = b^{(0)}(b^{(1)} \bullet a)\) for all \(a, b \in A_{par}\).
		
		In particular \(a = a^{(0)}(a^{(1)} \bullet 1_{A_{par}}) = a^{(0)} \varepsilon_{a^{(1)}}\). It follows that
		\begin{equation*}
			a^{(0)} (a^{(1)} \bullet (S(a^{(2)}) \bullet b)) \overset{\eqref{eq:leftAaction}}{=} a^{(0)} \varepsilon_{a^{(1)}} b = ab. \qedhere
		\end{equation*} 
	\end{proof}
	
	\begin{theorem}
		\label{thII:AparAparYD}
		Let \(H\) be a finite-dimensional Hopf algebra. Then \(\{A_{par}, A_{par}\}\) is a braided commutative algebra in \({_H}\mathcal{YD}^H\). Consequently, $H_{glob}$ is a Hopf algebroid.
	\end{theorem}
	\begin{proof}
		The second statement follows directly from the first one, by \cite[Theorem 4.1]{BM}.
		
		To prove the first statement, let us begin by defining an \(H\)-coaction on \(\{A_{par}, A_{par}\}\). The evaluation \(e_{AA} : A_{par} \otimes \{A_{par}, A_{par}\} \to A_{par},\ a \otimes \varphi \mapsto \varphi(a \otimes 1_H) \) and the half-braiding \(\beta_X\) from \cref{le:AYD} are morphisms of partial modules. Recall that \(\beta_X\) can be expressed using the coaction \(\rho_R\) \eqref{eqII:recall_rhoR}. Consider the composition
		\begin{align*}
			A_{par} \otimes H \otimes \{A_{par}, A_{par}\} \overset{\beta^{-1}_H \otimes \{A_{par}, A_{par}\}}&{\longrightarrow}  H \otimes A_{par} \otimes \{A_{par}, A_{par}\} \\
			\overset{H \otimes e_{AA}}&{\longrightarrow} H \otimes A_{par} \\
			\overset{\beta_H}&{\longrightarrow} A_{par} \otimes H, \\
			a \otimes h \otimes \varphi &\longmapsto \varphi(a^{(0)} \otimes 1_H)^{(0)} \otimes \varphi(a^{(0)} \otimes 1_H)^{(1)} S(a^{(1)}) h.
		\end{align*}
		It corresponds to a morphism of \(H\)-modules 
		\begin{align*}
			\xi : H \otimes \{A_{par}, A_{par}\} &\to \{A_{par}, A_{par} \otimes H\}, \\
			\xi(h \otimes \varphi)(a \otimes k) &= \varphi(a^{(0)} \otimes k_{(2)})^{(0)} \otimes \varphi(a^{(0)} \otimes k_{(2)})^{(1)} S(a^{(1)}) k_{(1)} h
		\end{align*} under the adjunction \eqref{eq:correspondence_righthom}. By evaluating \(\xi\) in \(1_H \otimes \varphi\) and composing with the isomorphism \(\{A_{par}, A_{par} \otimes H\} \cong \{A_{par}, A_{par}\} \otimes H\) from \cref{leII:Theta}, we obtain a map \[\rho : \{A_{par}, A_{par}\} \to \{A_{par}, A_{par}\} \otimes H.\] 
		We identify \(\{A_{par}, A_{par}\}\) with \({_{(H)}}\Hom(H, A_{par})\) using the isomorphism \(\Psi\) (see \cref{prop:Psibar} and \cref{cor:iso_Psi}), i.\,e.\ \(f \in \{A_{par}, A_{par}\}\) is identified with the partially linear map \(f(1_{A_{par}} \otimes -)\) from \(H\) to \(A_{par}\).
		Let \(\{y_j \mid j \in J\}\) be a basis of \(H\) and \(\{y^*_j \mid j \in J\}\) the dual basis of \(H^*\). Then 
		\[\rho : {_{(H)}}\Hom(H, A_{par}) \to {_{(H)}}\Hom(H, A_{par}) \otimes H, \quad f \mapsto f^{(0)} \otimes f^{(1)} = \sum_{j \in J} g_j \otimes y_j,\] 
		where the \(g_j : H \to A_{par}\) are defined by
		\[g_j(k) = f(k_{(2)})^{(0)} y^*_j(S^{-1}(k_{(3)}) f(k_{(2)})^{(1)} k_{(1)}).\]
		The codomain of \(\rho\) is \({_{(H)}}\Hom(H, A_{par}) \otimes H\) by construction, so \(g_j \in {_{(H)}}\Hom(H, A_{par})\) for all \(j \in J\).
		For any vector space \(V\), there is an obvious inclusion \[{_{(H)}}\Hom(H, A_{par}) \otimes V \subseteq \Hom_k(H, A_{par} \otimes V);\] evaluation of \(\sum_i f_i \otimes v_i\) in \(k \in H\) is given by \(\sum_i f_i(k) \otimes v_i\). This means in particular that \(\sum_i f_i \otimes v_i = \sum_i f'_i \otimes v'_i\) in \({_{(H)}}\Hom(H, A_{par}) \otimes V\) if and only if \(\sum_i f_i(k) \otimes v_i = \sum_i f'_i(k) \otimes v'_i\) for all \(k \in H\).
		
		Evaluating \(\rho(f)\) in \(k \in H\) gives
		\begin{equation}
			\label{eqII:coaction_explicit}
			f^{(0)}(k) \otimes f^{(1)} = \sum_{j \in J} g_j(k) \otimes y_j = f(k_{(2)})^{(0)} \otimes S^{-1}(k_{(3)}) f(k_{(2)})^{(1)} k_{(1)}.
		\end{equation}
		Then in particular 
		\begin{equation*}
			f^{(0)}(k) \epsilon(f^{(1)}) = f(k_{(2)})^{(0)} \epsilon(S^{-1}(k_{(3)}) f(k_{(2)})^{(1)} k_{(1)}) = f(k),
		\end{equation*}
		so \(f^{(0)} \epsilon(f^{(1)}) = f\) and we conclude that \(\rho\) is counital. Let us check that \(\rho\) is coassociative. We write
		\begin{equation*}
			\rho^2(f) = f^{(0)} \otimes f^{(1)} \otimes f^{(2)} = \sum_{i, j \in J} h_{ji} \otimes y_i \otimes y_j,
		\end{equation*}
		where \(h_{ji}(k) = g_j(k_{(2)})^{(0)} y^*_i(S^{-1}(k_{(3)}) g_j(k_{(2)})^{(1)} k_{(1)})\). For all \(k \in H\), we have
		\begin{align*}
			\sum_{i, j \in J} h_{ji}(k) &\otimes y_i \otimes y_j = \sum_{j \in J} g_j(k_{(2)})^{(0)} \otimes S^{-1}(k_{(3)}) g_j(k_{(2)})^{(1)} k_{(1)} \otimes y_j \\
			&= f(k_{(3)})^{(0)} \otimes S^{-1}(k_{(5)}) f(k_{(3)})^{(1)} k_{(1)} \otimes S^{-1}(k_{(4)}) f(k_{(3)})^{(2)} k_{(2)} \\
			&= f(k_{(2)})^{(0)} \otimes (S^{-1}(k_{(3)}) f(k_{(2)})^{(1)} k_{(1)})_{(1)} \otimes (S^{-1}(k_{(3)}) f(k_{(2)})^{(1)} k_{(1)})_{(2)} \\
			&= f^{(0)}(k) \otimes {f^{(1)}}_{(1)} \otimes {f^{(1)}}_{(2)},
		\end{align*}
		hence \(\rho^2(f) = f^{(0)} \otimes {f^{(1)}}_{(1)} \otimes {f^{(1)}}_{(2)}\).
		Next, we verify the compatibility condition
		\[(l \rightharpoonup f)^{(0)} \otimes (l\rightharpoonup f)^{(1)} = (l_{(2)} \rightharpoonup f^{(0)}) \otimes l_{(3)} f^{(1)} S^{-1}(l_{(1)})\]
		for left-right Yetter-Drinfeld modules. For any \(k, l \in H\)
		\begin{align*}
			(l_{(2)} &\rightharpoonup f^{(0)})(k) \otimes l_{(3)} f^{(1)} S^{-1}(l_{(1)}) =	\sum_{j \in J} (l_{(2)} \rightharpoonup g_j)(k) \otimes l_{(3)} y_j S^{-1}(l_{(1)}) \\ 
			&= \sum_{j \in J} g_j(kl_{(2)}) \otimes l_{(3)} y_j S^{-1}(l_{(1)})
			= f(k_{(2)} l_{(3)})^{(0)} \otimes l_{(5)} S^{-1}(k_{(3)}l_{(4)})f(k_{(2)}l_{(3)})^{(1)} k_{(1)} l_{(2)} S^{-1}(l_{(1)}) \\
			&=  f(k_{(2)} l)^{(0)} \otimes S^{-1}(k_{(3)})f(k_{(2)}l)^{(1)} k_{(1)} 
			= (l \rightharpoonup f)(k_{(2)})^{(0)} \otimes S^{-1}(k_{(3)})(l\rightharpoonup f)(k_{(2)})^{(1)} k_{(1)} \\
			&= (l \rightharpoonup f)^{(0)}(k) \otimes (l\rightharpoonup f)^{(1)}.
		\end{align*}
		This shows that \((l_{(2)} \rightharpoonup f^{(0)}) \otimes l_{(3)} f^{(1)} S^{-1}(l_{(1)}) = (l \rightharpoonup f)^{(0)} \otimes (l\rightharpoonup f)^{(1)}\) for all \(l \in H\).
		
		We showed that \(\{A_{par}, A_{par}\}\) is an object in \(\HYDH\). It is an \(H\)-module algebra for the \(H\)-linear composition \(\circ_{\{AAA\}}\), hence \(\{A_{par}, A_{par}\}\) is an algebra in \(\HYDH\) if it is an \(H^{op}\)-comodule algebra too. It is clear from \eqref{eqII:recall_rhoR} that \(\rho_R(ab) = a^{(0)} b^{(0)} \otimes b^{(1)} a^{(1)}\) for all \(a, b \in A_{par}\). Take \(f, \tilde{f} \in \{A_{par}, A_{par}\} \equiv {_{(H)}}\Hom(H, A_{par})\) and recall that the composition product on \(\{A_{par}, A_{par}\}\) coincides with the convolution product by \cref{le:Apar_convolution}, i.\,e.\ \((f \circ_{\{AAA\}}\tilde{f})(k) = f(k_{(1)}) \tilde{f}(k_{(2)})\). Now
		\begin{align*}
			(f \circ_{\{AAA\}} &\tilde{f})^{(0)}(k) \otimes (f \circ_{\{AAA\}} \tilde{f})^{(1)}
			\overset{\eqref{eqII:coaction_explicit}}{=} (f \circ_{\{AAA\}} \tilde{f})(k_{(2)})^{(0)} \otimes S^{-1}(k_{(3)}) (f \circ_{\{AAA\}} \tilde{f})(k_{(2)})^{(1)} k_{(1)} \\
			&= (f(k_{(2)}) \tilde{f}(k_{(3)}))^{(0)} \otimes S^{-1}(k_{(4)})  (f(k_{(2)}) \tilde{f}(k_{(3)}))^{(1)} k_{(1)} \\
			&= f(k_{(2)})^{(0)} \tilde{f}(k_{(3)})^{(0)} \otimes S^{-1}(k_{(4)}) \tilde{f}(k_{(3)})^{(1)} f(k_{(2)})^{(1)}  k_{(1)} \\
			&= f(k_{(2)})^{(0)} \tilde{f}(k_{(5)})^{(0)} \otimes S^{-1}(k_{(6)}) \tilde{f}(k_{(5)})^{(1)} k_{(4)} S^{-1}(k_{(3)}) f(k_{(2)})^{(1)} k_{(1)} \\
			&= f^{(0)}(k_{(1)}) \tilde{f}^{(0)}(k_{(2)}) \otimes \tilde{f}^{(1)} f^{(1)} \\
			&= (f^{(0)} \circ_{\{AAA\}} \tilde{f}^{(0)})(k) \otimes \tilde{f}^{(1)} f^{(1)}.
		\end{align*}
		
		Finally, we show that \(\{A_{par}, A_{par}\}\) is braided commutative. \\ For all \(k \in H,\) and \(f, \tilde{f} \in \{A_{par}, A_{par}\} \equiv {_{(H)}}\Hom(H, A_{par}),\) we have
		\begin{align*}
			(f^{(0)} &\circ_{\{AAA\}} (f^{(1)} \rightharpoonup \tilde{f}))(k) = f^{(0)}(k_{(1)}) \tilde{f}(k_{(2)}f^{(1)}) \\
			&= f(k_{(2)})^{(0)} \tilde{f}(k_{(4)} S^{-1}(k_{(3)}) f(k_{(2)})^{(1)} k_{(1)}) = f(k_{(2)})^{(0)} \tilde{f}(f(k_{(2)})^{(1)} k_{(1)}) \\
			\overset{\eqref{eqII:relation_ab}}&{=} f(k_{(2)})^{(0)} (f(k_{(2)})^{(1)} \bullet (S(f(k_{(2)})^{(2)}) \bullet \tilde{f}(f(k_{(2)})^{(3)} k_{(1)}) )) \\
			\overset{\eqref{eq:Psibar_image}}&{=} f(k_{(2)})^{(0)} (f(k_{(2)})^{(1)} \bullet \tilde{f}(S(f(k_{(2)})^{(2)})f(k_{(2)})^{(3)} k_{(1)})) \\
			&= f(k_{(2)})^{(0)} (f(k_{(2)})^{(1)} \bullet \tilde{f}( k_{(1)})) \overset{\eqref{eqII:Apar_commutative2}}{=} \tilde{f}(k_{(1)}) f(k_{(2)}) \\
			&= (\tilde{f} \circ_{\{AAA\}} f)(k). \qedhere
		\end{align*}
	\end{proof}

	\subsection{Morita equivalence}
	
	As mentioned before, the composition defined in \eqref{eq:compositionright} leads to the conclusion that the functor \(\{A_{par}, -\} : \HPMod \to \HMod\) lands in the category of left \(\{A_{par}, A_{par}\}\)-modules in \(\HMod,\) which is isomorphic to the category \(H_{glob}\)-modules. A natural question is then whether \(\{A_{par}, -\}\) is an equivalence of categories, that is, are $H_{par}$ and $H_{glob}$ Morita equivalent? Since the algebra $H_{par}$ is isomorphic to the partial smash product $A_{par} \underline{\#} H$ and $H_{glob}$ is defined as the global smash product $\{ A_{par} , A_{par} \} \# H$, the problem regarding their Morita equivalence goes back to the following result proven in \cite[Section 3]{ABenveloping}.
	
	\begin{theorem}
		\label{th:Morita}
		Let $A$ be a partial $H$-module algebra and $(B, \theta)$ a proper globalization of $A$ such that $\theta (A)$ is a two-sided ideal in $B$. Then the partial smash product $A\underline{\#} H$ is Morita equivalent to the global smash product $B\# H$.
	\end{theorem}
	
	Let us briefly recall the steps of the construction of a Morita context between $A\underline{\#} H$ and $B\# H$.
	First (see \cite[Lemma 4]{ABenveloping}), one considers a linear multiplicative map
	\[
	\Phi : A\otimes H  \rightarrow  B\# H,\quad a\otimes h  \mapsto  \theta (a) \otimes h.
	\]
	The map $\Phi$ restricted to the partial smash product
	\[
	A\underline{\#} H =(A\otimes H)(1_A \otimes 1_H) =\{  a (h_{(1)} \bullet 1_A)\otimes h_{(2)} \mid a \in A, h\in H\} 
	\]
	becomes a monomorphism of algebras. The elements of $A\underline{\#} H$ in the image of $\Phi$ are
	\[
	\theta (a) (h_{(1)} \btright \theta (1_A) ) \otimes h_{(2)}
	\]
	for \(a \in A\) and \(h \in H\).
	
	Define $M=\Phi (A\otimes H) \subseteq B\# H$ and
	\begin{equation}
		\label{eq:N}
		N = \left\langle (h_{(1)} \btright \theta (a)) \otimes h_{(2)} \mid a \in A, h \in H \right\rangle \subseteq B \# H.
	\end{equation}
	Now $M$ is an ($A\underline{\#} H$, $B\# H$)-bimodule and $N$ is a ($B\# H$, $A\underline{\#} H$)-bimodule (\cite[ Proposition 5 and Proposition 6]{ABenveloping}). All these structures are taken within the algebra $B\# H$, induced by the map $\Phi$. 
	The proof of the fact that $N$ is a ($B\# H$, $A\underline{\#} H$)-bimodule uses that \(\theta(A)\) is not just a right ideal, but a two-sided ideal in \(B\). It also relies on \(\theta(1_A)\) being a central idempotent in \(B\), but this follows from \(\theta(A) \trianglelefteq B\): for any \(b \in B,\) both \(b \theta(1_A)\) and \(\theta(1_A)b\) lie in \(\theta(A)\) and are equal to \(\theta(1_A)b\theta(1_A)\), because \(\theta(1_A)\) acts as a unit for \(\theta(A)\).

	Finally, the last ingredients to have a Morita context are the maps 
	\[
	\sigma :M\otimes_{B\# H} N \rightarrow \Phi (A\underline{\#} H) \quad \text{and} \quad \tau :N \otimes_{A\underline{\#} H} M \rightarrow B\# H.
	\]
	The first is a morphism of $A\underline{\#}H$-bimodules and the second is morphism of $B\# H$-bimodules. Both are given by the multiplication in $B\# H$. The associativity between these two maps follows from the associativity of the multiplication in $B\# H$. 
	
	In order to have a Morita equivalence, the maps $\sigma$ and $\tau$ need to be surjective. This is automatically satisfied in this case, because $MN=\Phi (A\underline{\#} H)$ and $NM=B\# H$ by \cite[Proposition 7]{ABenveloping}. The second statement uses implicitly that $B=H\btright \theta (A)$ and follows from the following identity in $B\# H$: for all \(a\in A\) and \(h,k \in H\)
	\[
	(h\btright \theta (a)) \otimes k= \left( (h_{(1)}\btright \theta (a)) \otimes h_{(2)}\right) \left( \theta (1_A ) \otimes S(h_{(3)})k \right).
	\]
	
	Before we apply \cref{th:Morita} to our dilation \(\{A_{par}, A_{par}\}\) of \(A_{par}\), we show that it is a globalization in the sense of \cref{def:globalization}.
	
	\begin{lemma}
		\label{le:multiplicative}
		The map \(\theta : A_{par} \to \{A_{par}, A_{par}\}\) defined by formula \eqref{eq:theta} is multiplicative. 
	\end{lemma}
	\begin{proof}
		For any \(a, b, c \in A_{par}\) and \(h \in H\), we have
		\[\theta(ab)(c \otimes h) = c (h \bullet (ab)) = c (h_{(1)} \bullet a) (h_{(2)} \bullet b) = \theta(a)(c \otimes h_{(1)})\, \theta(b)(1_{A_{par}} \otimes h_{(2)}).\]
		By \cref{le:Apar_convolution}, this last expression is equal to \((\theta(a) \circ_{\{AAA\}} \theta(b))(c \otimes h)\).
	\end{proof}
	
	\begin{proposition}
		The pair \((\{ A_{par} ,A_{par} \}, \theta)\) is a globalization of \(A_{par}\). 
	\end{proposition}
	\begin{proof}
		Take \(f \in \{A_{par}, A_{par}\}\). We start by showing that \(\theta(1_{A_{par}}) \circ_{\{AAA\}} f = \mathcal{T}(f)\), with \(\mathcal{T}\) as defined in formula \eqref{eq:calT}. 
		For any \(a \in A_{par}\) and \(h \in H,\) we have
		\begin{align*}
			(\theta(1_{A_{par}}) \circ_{\{AAA\}} f)(a \otimes h) &= \theta(1_{A_{par}})(a \otimes h_{(1)}) \, f(1_{A_{par}} \otimes h_{(2)}) = a (h_{(1)} \bullet 1_{A_{par}}) \, f(1_{A_{par}} \otimes h_{(2)}) \\
			&= a \varepsilon_{h_{(1)}} \, f(1_{A_{par}} \otimes h_{(2)}) \overset{(*)}{=} f(a \varepsilon_{h_{(1)}} \cdot (1_{A_{par}} \otimes h_{(2)})) \\
			\overset{\eqref{eq:AMX}}&{=} f(a\varepsilon_{h_{(1)}} \otimes h_{(2)}) = \mathcal{T}(f)(a \otimes h).
		\end{align*}
		In equality \((*)\), we used that since \(f\) is a morphism of partial modules, it is in particular \(A_{par}\)-linear.
		
		Next, we show that \(\theta(A_{par})\) is a right ideal in \(\{A_{par}, A_{par}\}\). Take $a=\varepsilon_{k^1} \ldots \varepsilon_{k^n}, b\in A_{par}$, $h\in H$ and $f\in \{ A_{par} , A_{par} \}$. 
		We calculate
		\begin{align*}
			(\theta (a) \ast f) (b\otimes h) &= \theta (a) (b\otimes h_{(1)}) \, f(1_{A_{par}} \otimes h_{(2)}) = b (h_{(1)} \bullet a)f(1_{A_{par}} \otimes h_{(2)}) \\
			&= bf((h_{(1)} \bullet a) \cdot (1_{A_{par}} \otimes h_{(2)})) \overset{\eqref{eq:AMX}}{=} bf((h_{(1)} \bullet a) \otimes h_{(2)}) =  bf(h\bullet (a\otimes 1_H)) \\
			&= b(h\bullet f(a\otimes 1_H ))  =  \theta (f(a\otimes 1_H)) (b\otimes h).
		\end{align*}
		We can conclude that \(\theta(A_{par})\) is a right ideal in \(\{A_{par}, A_{par}\}\).
		
		Together with \cref{propII:Apar_min_dilation} and \cref{le:multiplicative}, these computations prove that \((\{ A_{par} ,A_{par} \}, \theta)\) is a globalization of \(A_{par}\).
	\end{proof}
	
	\begin{theorem} \label{H_par_Morita_H_glob} 
		Suppose that the globalization $(\{ A_{par} ,A_{par} \}, \theta)$ of $A_{par}$ is proper, \\ i.\,e.\ $\{ A_{par} , A_{par} \} =H\rightharpoonup \theta (A_{par})$. Then the algebra
		\[
		H_{glob}= \{ A_{par} , A_{par} \} \# H
		\]
		is Morita equivalent to 
		\[
		H_{par} \cong A_{par} \underline{\#} H.
		\]
	\end{theorem}
	\begin{proof}
		In order to apply \cref{th:Morita}, we have to show that $\theta(A_{par})$ is a left ideal in $\{ A_{par} , A_{par} \}$. 
		Take \(f \in \{A_{par}, A_{par}\}, b \in A_{par}\) and \(h, k \in H\). Then
		\begin{align*}
			(f * \theta(\varepsilon_k))(b \otimes h) &= f(b \otimes h_{(1)}) \, \theta(\varepsilon_k)(1_{A_{par}} \otimes h_{(2)}) = b f(1_{A_{par}} \otimes h_{(1)}) (h_{(2)} \bullet \varepsilon_k) \\
			\overset{\eqref{eq:partial_module_Apar}}&{=} b f(1_{A_{par}} \otimes h_{(1)}) \varepsilon_{h_{(2)}k} \varepsilon_{h_{(3)}} = b f((1_{A_{par}} \otimes h_{(1)}) \cdot \varepsilon_{h_{(2)}k} \varepsilon_{h_{(3)}}) \\
			\overset{\eqref{eq:rightAaction}}&{=} b f(h_{(5)} \bullet (S^{-1}(h_{(4)}) \bullet (h_{(3)}k_{(2)} \bullet (S^{-1}(k_{(1)}) S^{-1}(h_{(2)}) \bullet (1_{A_{par}} \otimes h_{(1)}))))) \\
			\overset{\ref{PR5}}&{=} b f(h_{(3)} \bullet (k_{(2)} \bullet (S^{-1}(k_{(1)}) S^{-1}(h_{(2)}) \bullet (1_{A_{par}} \otimes h_{(1)})))) \\
			\overset{\ref{PR5}}&{=} b f(h_{(3)} \bullet (k_{(2)} \bullet (S^{-1}(k_{(1)}) \bullet (S^{-1}(h_{(2)}) \bullet (1_{A_{par}} \otimes h_{(1)}))))) \\
			\overset{\eqref{eq:AMX}}&{=} b f(h_{(3)} \bullet (k_{(2)} \bullet (S^{-1}(k_{(1)}) \bullet (\varepsilon_{S^{-1}(h_{(1)})} \otimes 1_H)))) \\
			\overset{\eqref{eq:rightAaction}}&{=} b f(h_{(3)} \bullet ((\varepsilon_{S^{-1}(h_{(1)})} \otimes 1_H) \cdot \varepsilon_k)) \\
			\overset{(\ref{eq:AMX}, \ref{eq:leftAaction})}&{=} b f(h_{(3)} \bullet (S^{-1}(h_{(2)}) \bullet (h_{(1)} \bullet ((1_{A_{par}} \otimes 1_H) \cdot \varepsilon_k)))) \\
			\overset{\ref{PR5}}&{=} b f(h \bullet ((1_{A_{par}} \otimes 1_H) \cdot \varepsilon_k)) = \theta(f((1_{A_{par}} \otimes 1_H) \cdot \varepsilon_k))(b \otimes h).
		\end{align*}
		Since \(A_{par}\) is generated by the elements \(\varepsilon_k,\)
		we obtain that
		\begin{equation}
			\label{eq:theta_leftideal}
			f * \theta(a) = \theta(f(1_{A_{par}} \otimes 1_H) a)
		\end{equation}
		for all \(f \in \{A_{par}, A_{par}\}\) and \(a \in A_{par}\), so \(\theta(A_{par})\) is a left ideal in \(\{A_{par}, A_{par}\}\).
		
		Since we suppose that the globalization $(\{ A_{par} ,A_{par} \}, \theta)$ is proper, \(H_{glob} = \{A_{par}, A_{par}\} \# H\) is Morita equivalent to \(H_{par} \cong A_{par} \underline{\#} H\) by \cref{th:Morita}.
	\end{proof}
	
	Recall that the category of modules over $H_{glob}= \{ A_{par} , A_{par} \} \# H$ is isomorphic to the category of modules over the algebra object $\{ A_{par} , A_{par} \}$ in $\HMod$, and the previous theorem immediately implies that this category is furthermore equivalent to
	the category ${_{H_{par}}\Mod} \cong \HPMod$ of partial $H$-modules. Our last result of this section shows that in case $H$ is finite-dimensional, this equivalence is realized exactly by the enrichment functor $\{A_{par},-\}$.

	\begin{proposition}
		\label{prop:NHpar}
		Let \(H\) be a finite-dimensional Hopf algebra and suppose that the globalization $(\{ A_{par} ,A_{par} \}, \theta)$ of $A_{par}$ is proper. Then
		\[\{A_{par}, -\} : \HPMod \to {_{\{A_{par}, A_{par}\}}}(\HMod)\]
		is an equivalence of categories.
	\end{proposition}
	\begin{proof}
		We start by showing that under the hypotheses of the proposition, the dilation functor \(D\) is isomorphic to \(\{A_{par}, -\}\).
		Since \((\{ A_{par} ,A_{par} \}, \theta)\) is proper, \(D(A_{par}) \cong \{A_{par}, A_{par}\}\) and \(\Xi_{A_{par}}\) \eqref{eqII:Xi} is an isomorphism. By \cite[Theorem 5.7]{ABVdilations}, \(D(A_{par} \otimes H) \cong D(A_{par}) \otimes H\) via the map 
		\[\zeta\left(\sum_i h_i \rightharpoonup \varphi(a_i \otimes k_i)\right) = \sum_i h_{i(1)} \rightharpoonup \varphi(a_i) \otimes h_{i(2)}k_i.\]
		By \cref{leII:Theta}, \(\{A_{par}, A_{par} \otimes H\} \cong \{A_{par}, A_{par}\} \otimes H\) too, via \(\Theta\). Since
		\begin{gather*}
			(\Theta \circ (\Xi_{A_{par}} \otimes H) \circ \zeta)\left(\sum_i h_i \rightharpoonup \varphi(a_i \otimes k_i)\right)(b \otimes l) = \sum_i (b \cdot (l_{(1)}h_{i(1)} \bullet a_i) \otimes l_{(2)}h_{i(2)} k_i \\
			\overset{(\ref{eq:AMX}, \ref{actegory1})}{=} b \cdot \sum_i lh_i \bullet (a_i \otimes k_i) = \Xi_{A_{par} \otimes H}\left(\sum_i h_i \rightharpoonup \varphi(a_i \otimes k_i)\right)(b \otimes l),
		\end{gather*}
		\(\Xi_{A_{par} \otimes H}\) is an isomorphism. For any partial module \(M\) and any \(f \in \{A_{par}, M\}\), we have \(f = \{A_{par}, f\}(\Id_{A_{par} \otimes H}),\) so 
		\[\Xi_M(D(f)(\Xi^{-1}_{A_{par} \otimes H}(\Id_{A_{par} \otimes H}))) = f\]
		by the naturality of \(\Xi\). This shows that \(\Xi_M\) is surjective, and hence an isomorphism by \cref{cor:Xi}.
		We conclude that \(D \cong \{A_{par}, -\}\). Since it is a right adjoint, it preserves kernels, so by \cite[Proposition 5.2]{ABVdilations}, \(D \cong \overline{H_{par}} \otimes_{H_{par}} -\). Here \(\overline{H_{par}}\) is an \((H, H_{par})\)-bimodule with actions
		\begin{equation}
			\label{eq:HHparbimodule}
			k \rightharpoonup (h \rightharpoonup \varphi(x)) \btleft y = (kh) \rightharpoonup \varphi(xy)
		\end{equation}
		for \(h, k \in H\) and \(x, y \in H_{par}\) (see \cite[Lemma 5.1]{ABVdilations}. The right \(H_{par}\)-module structure can be transferred to \(\{A_{par}, H_{par}\}\) using the \(H\)-linear isomorphism \(\Xi_{H_{par}} : \overline{H_{par}} \to \{A_{par}, H_{par}\}\); it is given by
		\((f \btleft y)(a \otimes h) = f(a \otimes h) y\) for \(f \in \{A_{par}, H_{par}\}, y \in H_{par}, a \in A_{par}\) and \(h \in H\). Hence \(\{A_{par}, -\} \cong D \cong \{A_{par}, H_{par}\} \otimes_{H_{par}} -\) as well. 
		
		We will exploit the Morita equivalence obtained in \cref{th:Morita} to prove the statement of the proposition. 
		We claim that the \((\{A_{par}, A_{par}\}\#H, H_{par})\)-bimodule \(N\) \eqref{eq:N} is isomorphic to \(\{A_{par}, H_{par}\}\). 
		Recall that \(N\) is generated as a vector space by the elements \[n = h_{(1)} \rightharpoonup \theta(a) \otimes h_{(2)} \in \{A_{par}, A_{par}\}\otimes H\] for \(a \in A_{par}\) and \(h \in H\). Applying the isomorphism \(\Theta\) \eqref{eq:Theta} to \(n\) we obtain the element in \(\{A_{par}, A_{par} \otimes H\}\) defined by
		\begin{equation*}
			\Theta(n)(b \otimes k) = (h_{(1)} \rightharpoonup \theta(a))(b \otimes k_{(1)}) \otimes k_{(2)}h_{(2)} = b(k_{(1)}h_{(1)} \bullet a) \otimes k_{(2)}h_{(2)}. 
		\end{equation*}
		We see that \(\Theta(n)(b \otimes k)\) lies in fact in \((A_{par} \otimes H)(1_{A_{par}} \otimes 1_H) = A_{par} \underline{\#} H \cong H_{par}\), so \(\Theta\) induces a well-defined and injective map \(\overline{\Theta} : N \to \{A_{par}, H_{par}\}\), where
		\begin{align}
			\overline{\Theta}(n)(b \otimes k) &= b(k_{(1)}h_{(1)} \bullet a)[k_{(2)}h_{(2)}] \overset{\eqref{eq:partial_module_Apar}}{=} b[k_{(1)}h_{(1)}] a [S(k_{(2)}h_{(2)})][k_{(3)}h_{(3)}] \nonumber\\
			\overset{(*)}&{=} b[k_{(1)}h_{(1)}] [S(k_{(2)}h_{(2)})][k_{(3)}h_{(3)}] a \overset{\eqref{eq:triple_bracket}}{=} b[kh]a = b (h \rightharpoonup \varphi(a))(k). \label{eq:Thetabar}
		\end{align}
		In equality \((*)\), we used that for all \(l \in H\), \(\tilde{\varepsilon_l} = [S(l_{(1)})][l_{(2)}]\) commutes with every element of \(A_{par}\) by \cite[Lemma 4.7(g)]{ABVparreps}.
		A similar computation shows that in \(\overline{H_{par}}\), we have the equality
		\begin{equation} \label{eq:equality_Hparbar} h \rightharpoonup \varphi(a[k]) = hk_{(2)} \rightharpoonup \varphi(S^{-1}(k_{(1)}) \bullet a). \end{equation}
		Since \(H_{par}\) is generated as a vector space by elements of the form \(a[k]\) for \(a \in A_{par}\) and \(k \in H,\) and any \(f \in \{A_{par}, H_{par}\} \cong \overline{H_{par}}\) is completely deterined by \(f(1_{A_{par}} \otimes k)\) (\cref{prop:Psibar}), \eqref{eq:Thetabar} and \eqref{eq:equality_Hparbar} together show that \(\overline{\Theta}\) surjective too.
		
		Take \(f \in \{A_{par}, A_{par}\}\) and \(k \in H\). By \cite[Proposition 5]{ABenveloping}, 
		\begin{align*} 
			(f \# k) \cdot n &= k_{(2)} h_{(2)} \rightharpoonup \left( (S^{-1}(k_{(1)}h_{(1)}) \rightharpoonup f) * \theta(a)\right) \otimes k_{(3)}h_{(3)} \\
			\overset{\eqref{eq:theta_leftideal}}&{=} k_{(2)} h_{(2)} \rightharpoonup \theta(f(1_{A_{par}} \otimes S^{-1}(k_{(1)} h_{(1)}))a) \otimes k_{(3)}h_{(3)}.
		\end{align*}
		Hence
		\begin{align*}
			\overline{\Theta}((f \# k) \cdot n)(c \otimes l) &= c [lk_{(2)}h_{(2)}] f(1_{A_{par}} \otimes S^{-1}(k_{(1)}h_{(1)})) a \\
			\overset{\eqref{eq:triple_bracket}}&{=} c [l_{(2)}k_{(3)}h_{(3)}]\varepsilon_{S^{-1}(l_{(1)} k_{(2)} h_{(2)})} f(1_{A_{par}} \otimes S^{-1}(k_{(1)}h_{(1)})) a \\
			\overset{\eqref{eq:AMX}}&{=} c [l_{(4)}k_{(3)}h_{(3)}] f(\varepsilon_{S^{-1}(l_{(3)} k_{(2)} h_{(2)})} \otimes S^{-1}(l_{(2)}k_{(1)}h_{(1)})l_{(1)}) a \\
			&= c [l_{(3)}k_{(2)}h_{(2)}] (S^{-1}(l_{(2)}k_{(1)}h_{(1)}) \bullet f(1_{A_{par}} \otimes l_{(1)})) a \\
			&= c [l_{(4)}k_{(3)}h_{(3)}] [S^{-1}(l_{(3)}k_{(2)}h_{(2)})] f(1_{A_{par}} \otimes l_{(1)})) [l_{(2)}k_{(1)}h_{(1)}] a \\
			&= c  f(1_{A_{par}} \otimes l_{(1)})) [l_{(4)}k_{(3)}h_{(3)}] [S^{-1}(l_{(3)}k_{(2)}h_{(2)})] [l_{(2)}k_{(1)}h_{(1)}] a \\
			\overset{\eqref{eq:triple_bracket}}&{=}  f(c \otimes l_{(1)})) [l_{(2)}kh] a = \overline{\Theta}(n)(f(c \otimes l_{(1)}) \otimes l_{(2)}k) \\
			&= (f \circ_{\{AAH_{par}\}} (k \rightharpoonup \overline{\Theta}(n)))(c \otimes l).
		\end{align*}
		This shows that \(\overline{\Theta}\) is a morphism of \(\{A_{par}, A_{par}\} \# H\)-modules.
		
		Take now \(y = b[k] \in H_{par}\). By \cite[Proposition 6]{ABenveloping}, \(n \btleft (b \# k) = h_{(1)}k_{(2)} \rightharpoonup \theta(S^{-1}(k_{(1)}) \bullet (ab)) \otimes h_{(2)}k_{(3)}\). Hence
		\begin{align*}
			\overline{\Theta}(n \btleft (b \# k))(c \otimes l) &= c(hk_{(2)} \rightharpoonup \varphi(S^{-1}(k_{(1)}) \bullet (ab)))(l) \\
			\overset{\eqref{eq:equality_Hparbar}}&{=} c(h \rightharpoonup \varphi(ab[k]))(l) = (\overline{\Theta}(n) \btleft y)(c \otimes l).
		\end{align*}
		We conclude that \(\overline{\Theta}\) is right \(H_{par}\)-linear as well. 
		
		By \cref{th:Morita}, the functor \(N \otimes_{H_{par}} - : {_{H_{par}}}\Mod \to {_{\{A_{par}, A_{par}\}\# H}}\Mod\) is an equivalence of categories. We showed that \(N \otimes_{H_{par}} - \cong \{A_{par}, H_{par}\} \otimes_{H_{par}} - \cong \{A_{par}, -\}\), so \(\{A_{par}, -\} : \HPMod \to {_{\{A_{par}, A_{par}\}}}(\HMod)\) is an equivalence as well.
	\end{proof}
	
	The previous theorem applies in particular to finite-dimensional pointed Hopf algebras over a field of characteristic zero thanks to \cref{th:Xi_iso_pointed}.

	\begin{remark}
		Let us put the above result in relation with known general results on module categories over (finite) tensor categories. For example, \cite[Theorem 7.10.1]{EGNObook} provides sufficient conditions on such a finite module category to arise as a category of modules over an internal algebra object. However, to apply this result directly to our context, we would not only need to require that \(\{A_{par}, -\}\) is exact (this is certainly the case if $H$ is a finite-dimensional pointed Hopf algebra over a field of characteristic zero by \cref{th:Xi_iso_pointed}) but also, and more importantly, that $A_{par}$ is finite-dimensional. The only known examples of Hopf algebras satisfying both conditions are finite group algebras. Therefore, our results allow to go beyond the setting of \cite{EGNObook}. 
		
		Recently an alternative approach to overcome the finiteness conditions from \cite{EGNObook} has been proposed in \cite{MattiTony}.
		Suppose that \(\{A_{par}, -\} = \Hom_{\HPMod}(A_{par} \otimes H, -)\) is exact, which is equivalent to \(A_{par} \otimes H\) being a projective \(H_{par}\)-module. By \cref{leII:projtoHpar}, \(A_{par} \otimes H\) is a generator in \(\HPMod,\) and it follows that for any projective \(H\)-module \(P\), the partial module \(A_{par} \otimes P\) is projective. Moreover, for any partial module \(M\), there is a global module \(X\) and a surjection \(A_{par} \otimes X \to M\) (it suffices to take \(X = \bigoplus_I H\) for a generating set \(I\) of \(M\)). This means that \(A_{par}\) is an \textit{\(\HMod\)-projective \(\HMod\)-generator} in the sense of \cite[Definition 4.2]{MattiTony}. Since \(A_{par} \otimes -\) has a right adjoint \(\{A_{par}, -\},\) the right actegory version of \cite[Theorem 5.1]{MattiTony} can be applied: we conclude that \(\{A_{par}, -\}\) provides an equivalence between \(\HPMod\) and the Eilenberg-Moore category of the monad \(\{A_{par}, A_{par} \otimes -\}\) on \(\HMod\).
		
		If \(H\) is moreover finite-dimensional, then we can see that \(\{A_{par}, A_{par} \otimes -\} \cong \{A_{par}, A_{par}\} \otimes -\) as monads on \(\HMod\) and we recover \cref{prop:NHpar}. Indeed, \(A_{par} \otimes H\) is a finitely generated \(H_{par}\)-module because \(A_{par}\) is cyclic and \(H\) is finite-dimensional. It follows that \(\{A_{par}, -\} = \Hom_{\HPMod}(A_{par} \otimes H, -)\) preserves direct sums, and so do \(\{A_{par}, A_{par} \otimes -\}\) and \(\{A_{par}, A_{par}\} \otimes -\). By \cref{leII:Theta}, we know that \(\{A_{par}, A_{par} \otimes H\} \cong \{A_{par}, A_{par}\}\otimes H\). Then by a classical argument, we find that also \(\{A_{par}, A_{par} \otimes X\} \cong \{A_{par}, A_{par}\}\otimes X\) for any \(H\)-module \(X\), as any such \(X\) can be represented as a coequalizer 
		\[\begin{tikzcd}
			H \otimes H \otimes X \arrow[shift left = 4pt]{r}{H \otimes\, \btright} \arrow[shift right = 4pt, swap]{r}{\mu \otimes X} & H \otimes X \arrow{r}{\btright} & X
		\end{tikzcd}\]
		of free \(H\)-modules, and both functors \(\{A_{par}, A_{par} \otimes -\}\) and \(\{A_{par}, A_{par}\} \otimes -\) are exact.
	\end{remark}
	
	\subsection{The group case}
	Let us look at the group algebra \(H = kG\) for a finite group \(G\) in more detail.
	The category of $\{ A_{par} ,A_{par} \}$-modules in the category ${}_{\Bbbk G} \Mod$ is isomorphic to the category of modules over $(\Bbbk G)_{glob}=\{ A_{par} ,A_{par} \} \# \Bbbk G$. We are going to show a more explicit construction of $(\Bbbk G)_{glob}$.
	
	For a finite group $G$, the subalgebra \(A_{par}\) of \(k_{par}G\) is generated by the commuting idempotents \(\varepsilon_g = [g][g^{-1}]\) for \(g \in G\), and it has a basis of orthogonal idempotents
	\[
	P_X = \prod_{g \in G} \varepsilon_g \prod_{\bar{g} \notin G} (1 - \varepsilon_{\bar{g}})
	\]
	for \(X \in \Pp_1(G)\), where $\Pp_1 (G)$ is the set of subsets of $G$ containing the neutral element $e\in G$. The partial action of \(kG\) on \(A_{par}\) in terms of the elements \(P_X\) is then
	\[
	g \bullet P_X = \begin{cases}
		P_{gX} &\text{ if } g^{-1} \in X, \\
		0 &\text{ if } g^{-1} \notin X.
	\end{cases}
	\]
	Moreover, in the case of a finite group $G$, the algebra $k_{par}G$ is isomorphic to a groupoid algebra (see \cite{DEP}). More specifically, $\Bbbk_{par} G \cong \Bbbk \Gamma (G)$, where $\Gamma (G)$ is the groupoid 
	\[
	\Gamma (G) =\{ (g,A) \in G \times \mathcal{P}(G) \; | \; e,g^{-1} \in A \} ,
	\]
	whose set of objects is
	\[
	\Gamma (G)^{(0)} =\{ (e,A) \in G \times \mathcal{P}(G) \; | \; e \in A \} ,
	\]
	and whose structure is given by
	\begin{eqnarray}
		& \, & s(g,A)=(e,A), \qquad t(g,A)=(e,gA), \nonumber \\
		& \, & (g,A)(h,B)= \begin{cases} (gh,B) & \text{ if } A=hB \\
			\text{not defined} &\text{ if } A\neq hB \end{cases} \label{eq:multiplication_Gamma}\\
		& \, & (g,A)^{-1} =(g^{-1} ,gA). \nonumber
	\end{eqnarray}
	
	Let us characterize the algebra $(\Bbbk G)_{glob}$ as a groupoid algebra in the same way we did for $\Bbbk_{par} G$. Denote by \(\Pp'(G)\) the set of nonempty subsets of \(G\), which is a semilattice with respect to the union, \(\cup\). In particular it is an inverse semigroup and we can consider its semigroup algebra \(\Bb(G) \coloneqq k\Pp'(G)\). By \cite[Theorem 4.2]{Steinberg}, \(\Bb(G)\) is isomorphic to the groupoid algebra of the underlying groupoid of \(\Pp'(G)\). This underlying groupoid has only identity morphisms because every element in \(\Pp'(G)\) is idempotent. Hence \(\Bb(G)\) has a basis of orthogonal idempotents \(\{Q_X \mid X \in \Pp'(G)\}\). In fact \(\Bb(G)\) is even a \(kG\)-module algebra by the action
	\[g \triangleright Q_X = Q_{gX}\]
	for all \(g \in G\) and \(\varnothing \neq X \subseteq G\). We claim that the standard dilation of \(A_{par}\) is isomorphic to \(\Bb(G)\). To this aim, let us observe first that for any \(X \in \Pp'(G)\) and \(x \in X,\)
	\[(x \rightharpoonup \varphi(P_{x^{-1}X}))(g) = gx \bullet P_{x^{-1}X} = \begin{cases}
		P_{gX} &\text{ if } g^{-1} \in X, \\
		0 &\text{ if } g^{-1} \notin X,
	\end{cases}\]
	which is clearly independent of the choice of \(x \in X\). 
	
	Moreover, the set \(\{x \rightharpoonup \varphi(P_{x^{-1}X}) \mid X \in \Pp'(G)\}\) is a basis for \(\overline{A_{par}}\): it is generating because \(\sum_i g_i \rightharpoonup \varphi(P_{Y_i}) = \sum_i g_i \rightharpoonup \varphi(P_{g_i^{-1} g_iY_i})\) for any \(g_i \in G, Y_i \in \Pp_1(G)\); we have that \(g_i \in g_iY_i\) because \(Y_i\) contains the unit of \(G\). It is free too, since if there exist \(\alpha_X \in k\) for \(X \in \mathcal{P}'(G)\) such that for all \(h \in G\)
	\begin{align*}
		0 &= \sum_{X \in \Pp'(G)} \alpha_X (x \rightharpoonup \varphi(P_{x^{-1}X}))(h) = \sum_{X \in \Pp'(G)} \alpha_X hx \bullet P_{x^{-1}X} \\
		&= \sum_{X \ni h^{-1}} \alpha_X P_{hX} = \sum_{X \in \Pp_1(G)} \alpha_{h^{-1}X} P_X,
	\end{align*}
	then \(\alpha_{h^{-1}X} = 0\) for all \(X \in \Pp_1(G)\) and \(h \in G,\) i.\,e.\ \(\alpha_X = 0\) for all \(X \in \Pp'(G)\).
	
	It follows that the assignment
	\begin{align*}
		\zeta : \Bb(G) \to \overline{A_{par}}, \quad Q_X \mapsto \left(x \rightharpoonup \varphi(P_{x^{-1}X})\right)
	\end{align*}
	is a well-defined linear isomorphism. Let us show that it is \(kG\)-linear. Take \(g, h \in G, X \in \Pp'(G)\) and let \(x \in X\). Then \(gx \in gX\), so
	\begin{align*}
		\zeta(g \triangleright Q_X)(h) = \zeta(Q_{gX})(h) = (gx \rightharpoonup\varphi(P_{(gx)^{-1}gX}))(h) = (g \rightharpoonup (x \rightharpoonup \varphi(P_{x^{-1}X})))(h).
	\end{align*} 
	It is multiplicative too, because the elements \(x \rightharpoonup \varphi(P_{x^{-1}X})\) are orthogonal idempotents in \(\overline{A_{par}}\):
	\begin{equation*}
		(x \rightharpoonup \varphi(P_{x^{-1}X}))(y \rightharpoonup \varphi(P_{y^{-1}Y})) \overset{\eqref{eq:multiplication_globalization}}{=} x \rightharpoonup \varphi(P_{x^{-1}X} (x^{-1}y \bullet P_{y^{-1}Y})) 
		\overset{\phantom{\eqref{eq:multiplication_globalization}}}{=} \begin{cases}
			x \rightharpoonup \varphi(P_{x^{-1}X}) &\text{ if } X = Y, \\
			0 & \text{ if } X \neq Y.
		\end{cases}
	\end{equation*}
	We conclude that \(\zeta\) is an isomorphism of \(kG\)-module algebras. Since by \cref{th:Xi_iso_pointed} \(\overline{A_{par} (G)} \cong \{A_{par} (G), A_{par} (G)\}\), we obtain that
	\[(kG)_{glob} = \{A_{par}(G), A_{par}(G)\} \# kG \cong \overline{A_{par}(G)} \# \Bbbk G \cong  \Bb(G) \# kG .\]
	The smash product on the right hand side is the groupoid algebra of the groupoid \(\Gamma_{glob}(G),\) which has set of objects \(\Pp'(G)\) and morphisms
	\(G \times \Pp'(G)\); the composition is defined similarly to \eqref{eq:multiplication_Gamma}:
	\begin{equation*}
		(g, A) \cdot (h, B) = \begin{cases}
			(gh, B) &\text{ if } A = hB; \\
			\text{not defined} & \text{ if } A \neq hB,
		\end{cases}
	\end{equation*}
	for all \(g, h \in G\) and \(A, B \in \Pp'(G)\). This is exactly the description of the algebra \(k_{glob}G\) introduced in \cite[Chapter 3]{Velascothesis}, hence \((kG)_{glob} \cong k\Gamma_{glob}(G) = k_{glob}G\), and \((kG)_{glob}\) is in fact a weak Hopf algebra (because it is isomorphic to a groupoid algebra).
	
	In conclusion, we obtained the following chain of isomorphisms of algebras.
	
	\begin{corollary}
		For a finite group $G$, we have the algebra isomorphisms 
		\[
		\Bbbk_{glob} G\cong \Bbbk \Gamma_{glob}(G) \cong \mathcal{B} (G) \# \Bbbk G \cong \overline{A_{par}(G)}\# \Bbbk G \cong \{ A_{par} , A_{par} \} \# \Bbbk G \cong (\Bbbk G)_{glob} .
		\]
	\end{corollary}
	
	\section*{Acknowledgements}	
	William Hautekiet would like to thank Eliezer Batista for the hospitality during his visit to the Universidade Federal de Santa Catarina in February 2024, when this work started.
	
	Joost Vercruysse would like to thank the Fédération Wallonie-Bruxelles for support via the ARC ``From algebra to combinatorics, and back'' and the FNRS for support via the PDR T.0318.25 ``Redisclosure''.

\end{document}